\documentclass[12pt]{amsart}

\usepackage[matrix,arrow,curve,frame]{xy}
\usepackage{amsmath,amsthm,amssymb,enumerate}
\usepackage{latexsym}
\usepackage{amscd}
\usepackage[colorlinks=false]{hyperref}
\usepackage{euscript}

\newtheorem{thm}[subsection]{Theorem}

\newtheorem{cor}[subsection]{Corollary}
\newtheorem{lemma}[subsection]{Lemma}

\newtheorem{remark}[subsection]{Remark}

\theoremstyle{definition}
\newtheorem{example}[subsection]{Example}

\numberwithin{equation}{section}

\def\ZZ{{\mathbb Z}}

\def\cO{{\cal O}}

\def\cP{{\cal P}}
\def\cA{{\cal A}}

\def\ra{\rightarrow}
\def\bra{\langle}
\def\ket{\rangle}

\def\cA{{\mathcal A}}
\def\cB{{\mathcal B}}

\def\cF{{\mathcal F}}

\def\cH{{\mathcal H}}
\def\cI{{\mathcal I}}
\def\cJ{{\mathcal J}}

\def\cM{{\mathcal M}}
\def\cN{{\mathcal N}}
\def\cO{{\mathcal O}}
\def\cP{{\mathcal P}}

\def\cR{{\mathcal R}}
\def\cS{{\mathcal S}}

\def\cV{{\mathcal V}}
\def\cW{{\mathcal W}}

\def\gg{{\mathfrak g}}

\def\gl{{\mathfrak l}}

\def\go{{\mathfrak o}}
\def\gp{{\mathfrak p}}

\def\gs{{\mathfrak s}}

\def\ch{{\text{ch}}}

\newfont{\german}{eufm10}

\begin{document}
\pagestyle{plain}

\title
{Orbifolds of symplectic fermion algebras}

\author{Thomas Creutzig and Andrew R. Linshaw}
\address{Department of Mathematics, University of Alberta}
\email{creutzig@ualberta.ca}
\address{Department of Mathematics, University of Denver}
\email{andrew.linshaw@du.edu}


{\abstract We present a systematic study of the orbifolds of the rank $n$ symplectic fermion algebra $\cA(n)$, which has full automorphism group $Sp(2n)$. First, we show that $\cA(n)^{Sp(2n)}$ and $\cA(n)^{GL(n)}$ are $\cW$-algebras of type $\cW(2,4,\dots, 2n)$ and $\cW(2,3,\dots, 2n+1)$, respectively. Using these results, we find minimal strong finite generating sets for $\cA(mn)^{Sp(2n)}$ and $\cA(mn)^{GL(n)}$ for all $m,n\geq 1$. We compute the characters of the irreducible representations of $\cA(mn)^{Sp(2n)\times SO(m)}$ and $\cA(mn)^{GL(n)\times GL(m)}$ appearing inside $\cA(mn)$, and we express these characters using partial theta functions. Finally, we give a complete solution to the {\it Hilbert problem} for $\cA(n)$; we show that for any reductive group $G$ of automorphisms, $\cA(n)^G$ is strongly finitely generated.}

\maketitle
\section{Introduction}

Vertex algebras are a fundamental class of algebraic structures that arose out of conformal field theory and have applications in a diverse range of subjects. Given a vertex algebra $\cV$ and a group $G$ of automorphisms of $\cV$, the invariant subalgebra $\cV^G$ is called an {\it orbifold} of $\cV$. Many interesting vertex algebras can be constructed either as orbifolds or as extensions of orbifolds. A spectacular example is the {\it Moonshine vertex algebra} $V^{\natural}$, whose full automorphism group is the Monster, and whose graded character is $j(\tau) - 744$ where $j(\tau)$ is the modular invariant $j$-function \cite{FLM}.

In physics, {\it rational} vertex algebras, of which $V^{\natural}$ is an example, correspond to rational two-dimensional conformal field theories. Other well-known examples include the Virasoro minimal models \cite{GKO,WaI}, affine vertex algebras at positive integer level \cite{FZ}, lattice vertex algebras associated to positive-definite even lattices \cite{D}, and certain families of $\cW$-algebras \cite{ArI,ArII}. A rational vertex algebra $\cV$ has only finitely many irreducible, admissible modules, and any admissible $\cV$-module is completely reducible. On the other hand, many interesting vertex algebras admit modules which are reducible but indecomposable. Following \cite{CR}, we shall call such vertex algebras {\it logarithmic}, as the corresponding conformal field theories often have logarithmic singularities in their correlation functions. For many years after Zhu's thesis appeared in 1990 \cite{Zh}, it was believed that rationality and the $C_2$-cofiniteness conditions were equivalent; see for example \cite{DLMII,ABD}. However, this was disproven by the construction of the $\cW_{p,q}$-triplet algebras \cite{KI,FGST}, which are logarithmic and $C_2$-cofinite \cite{A,AM,TW}. 

One of the first logarithmic conformal field theories studied in the physics literature was the {\it symplectic fermion} theory \cite{KII}. The symplectic fermion algebra $\cA(n)$ of rank $n$ has odd generators $e^i, f^i$ for $i=1,\dots, n$ satisfying operator product expansions $$e^i(z) f^j(w) \sim \delta_{i,j} (z-w)^{-2}.$$ It is an important example of a free field algebra, and the simplest triplet algebra $\cW_{2,1}$ coincides with the orbifold $\cA(1)^{\mathbb{Z}/2\mathbb{Z}}$. It is also known that $\cA(n)^{\mathbb{Z}/2\mathbb{Z}}$ is logarithmic and $C_2$-cofinite \cite{A}, but little is known about the structure and representation theory of more general orbifolds of $\cA(n)$.

Our goal in this paper is to carry out a systematic study of the orbifolds $\cA(n)^G$, where $G$ is any reductive group of automorphisms of $\cA(n)$. The full automorphism group of $\cA(n)$ is the symplectic group $Sp(2n)$. In order to describe $\cA(n)^G$ for a general $G$, a detailed understanding of the structure and representation theory of $\cA(n)^{Sp(2n)}$ is necessary. As we shall see, all irreducible, admissible $\cA(n)^{Sp(2n)}$-modules are highest-weight modules, and $\cA(n)^G$ is completely reducible as an $\cA(n)^{Sp(2n)}$-module.

Our study of $\cA(n)^{Sp(2n)}$ is based on {\it classical invariant theory}, and follows the approach developed in \cite{LI,LII,LIII,LIV,LV}. First, $\cA(n)$ admits an $Sp(2n)$-invariant filtration such that $\text{gr}(\cA(n))\cong \bigwedge \bigoplus_{k\geq 0} U_k$ as supercommutative rings. Here each $U_k$ is a copy of the standard $Sp(2n)$-representation $\mathbb{C}^{2n}$. As a vector space, $\cA(n)^{Sp(2n)}$ is isomorphic to $R = \big(\bigwedge \bigoplus_{k\geq 0} U_k\big)^{Sp(2n)}$, and we have isomorphisms of graded commutative rings $$\text{gr}(\cA(n)^{Sp(2n)}) \cong \text{gr}(\cA(n))^{Sp(2n)} \cong R.$$ In this sense, we regard $\cA(n)^{Sp(2n)}$ as a {\it deformation} of the classical invariant ring $R$.

By an odd analogue of Weyl's first and second fundamental theorems of invariant theory for the standard representation of $Sp(2n)$, $R$ is generated by quadratics $$\{q_{a,b}|\ 0\leq a\leq b\},$$ and the ideal of relations among the $q_{a,b}$'s is generated by elements $$\{p_I|\ I = (i_0,  \dots, i_{2n+1}),\ 0\leq i_0 \leq \cdots \leq i_{2n+1}\}$$ of degree $n+1$, which are analogues of Pfaffians. We obtain a corresponding strong generating set $\{\omega_{a,b}\}$ for $\cA(n)^{Sp(2n)}$, as well as generators $\{P_I\}$ for the ideal of relations among the $\omega_{a,b}$'s, which correspond to the classical relations with suitable quantum corrections. In fact, there is a more economical strong generating set $\{j^{2k} = \omega_{0,2k}|\ k\geq 0\}$, and the sets $\{\omega_{a,b}|\ 0\leq a\leq b\}$ and $\{\partial^i j^{2k}|\ i,k\geq 0\}$ are related by a linear change of variables. The relation of minimal weight among the generators occurs at weight $2n+2$, and corresponds to $I = (0,0,\dots,0)$.

A key technical result in this paper is the analysis of the quantum corrections of the above classical relations. For each $p_I$, there is a certain correction term $R_I$ appearing in $P_I$ which we call the {\it remainder}. We show that $R_I$ satisfies a recursive formula which implies that the relation of minimal weight has the form \begin{equation} \label{introdecoup} j^{2n} = Q(j^0, j^2,\dots, j^{2n-2}).\end{equation} Here $Q$ is a normally ordered polynomial in $j^0, j^2,\dots j^{2n-2}$ and their derivatives. We call \eqref{introdecoup} a {\it decoupling relation}, and by applying the operator $j^2 \circ_1$ repeatedly, we can construct higher decoupling relations $$j^{2m} = Q_m(j^0, j^2, \dots, j^{2n-2})$$ for all $m>n$. This shows that $\{j^0, j^2, \dots, j^{2n-2}\}$ is a minimal strong generating set for $\cA(n)^{Sp(2n)}$, and in particular $\cA(n)^{Sp(2n)}$ is of type $\cW(2,4, \dots, 2n)$; see Theorem \ref{maincor}. Using this result, we give a minimal strong generating set for $\cA(mn)^{Sp(2n)}$ 
for 
all $m,n\geq 1$. Similarly, using the invariant theory of $GL(n)$ we show that $\cA(n)^{GL(n)}$ is of type $\cW(2,3,\dots, 2n+1)$, and we give a minimal strong generating set for $\cA(mn)^{GL(n)}$ for all $m,n\geq 1$. 

\subsection{Character decompositions} By a general theorem of Kac and Radul \cite{KR}, $\cA(mn)$ decomposes as a sum of irreducible $\cA(mn)^{G}$-modules for every reductive group $G\subset Sp(2mn)$ (see also \cite{DLMI}). Such decompositions are very useful for computing the characters of orbifolds of free field algebras; see for example \cite{WaII,KWY}. More recently, the characters of some orbifolds of the $\beta\gamma$-system of rank $n$ were computed in \cite{BCR}. The task was to find Fourier coefficients of certain negative index meromorphic Jacobi forms, and one way to solve the decomposition problem used the denominator identity of affine Lie superalgebras. We will instead use the denominator identity of the finite-dimensional Lie superalgebra $\gs\gp\go(2n|m)$ to find the character decomposition when $G=Sp(2n)\times SO(m)$, and the denominator identity of $\gg\gl(n|m)$ to find the character decomposition for $G=GL(n)\times GL(m)$. In all these cases we obtain explicit character formulas, and they are expressed using partial theta functions. The character of $\cA(mn)$ is
\[
q^{\frac{mn}{12}}\prod_{j=1}^\infty (1+q^j)^{2mn} = \left(\frac{\eta\left(q^2\right)}{\eta\left(q\right)}\right)^{2mn}
\]
with the Dedekind eta function 
$\eta(q)=q^{\frac{1}{24}}\prod\limits_{n=1}^\infty (1-q^n)$. 

In order to perform the character decomposition, we need to refine the grading. 
For this let $\gg_0=\gg\gl(m)\oplus \gg\gl(n)$ or $\gg_0=\gs\gp(2m)\oplus\gs\go(n)$ be the even subalgebra of
$\gg=\gg\gl(m|n)$, respectively $\gg=\gs\gp\go(2m|n)$. Then the set of odd roots $\Delta_1$ of $\gg$ is contained in the weight lattice of $\gg_0$, and the $\cA(mn)$-algebra character, graded by both the weight lattice of $\gg_0$ and by conformal dimension, is
\begin{equation}\nonumber
\ch[\cA(mn)]= q^{\frac{mn}{12}} \prod_{\alpha\in\Delta_1} \prod_{n=1}^\infty \left(1+e^{\alpha}q^n\right).
\end{equation}
The character can be rewritten as 
\begin{equation}\nonumber
\begin{split}
\ch[\cA(mn)]&= e^{-\rho_1}q^{\frac{mn}{12}} \prod_{\alpha\in\Delta^+_1} \frac{e^{\frac{\alpha}{2}}}{\left(1+e^{-\alpha}\right)}\prod_{n=1}^\infty \left(1+e^{\alpha}q^n\right)\left(1+e^{-\alpha}q^{n-1}\right)
= \frac{1}{e^{\rho_1}}\prod_{\alpha\in\Delta^+_1} \frac{\vartheta\left(e^\alpha ; q\right)}{\left(1+e^{-\alpha}\right)\eta(q)},
\end{split}
\end{equation}
where $\rho_1$ is the odd Weyl vector of $\gg$, $\Delta_1^+$ is the set of its positive odd roots and
\begin{equation}\nonumber 
\vartheta(z; q) = \sum_{n\in\ZZ} z^{n+\frac{1}{2}}q^{\frac{1}{2}\left(n+\frac{1}{2}\right)^2}
\end{equation}
is a standard Jacobi theta function. 
Let $P^+$ be the set of dominant weights of the even subalgebra $\gg_0$. 
Then our result, Theorem \ref{thm:ch}, states that the graded character of $\cA(mn)$ decomposes as 
\[
\ch[\cA(mn)] = \sum_{\Lambda\in L\cap P^+} \text{ch}_{\Lambda}  B_\Lambda.
\]
Here $\text{ch}_{\Lambda}$ is the irreducible highest-weight representation of $\gg_0$ of highest-weight $\Lambda$.
The branching function is
\[
B_\Lambda = \frac{1}{\eta(q)^{|\Delta^+_1|}}\frac{|W^\sharp|}{|W|}\sum_{\omega\in W}\epsilon(w) \sum_{\left((n_\alpha),(m_\beta)\right) \in I_{\omega(\Lambda+\rho_0)-\rho_0}}  q^{\frac{1}{2}\left(n_\alpha+\frac{1}{2}\right)^2} P_{m_\beta}(q)
\]
and
\[
P_n(q):= q^{\frac{1}{2}\left(n+\frac{1}{2}\right)^2}\sum\limits_{m=0}^\infty (-1)^{m} q^{\frac{1}{2}\left(m^2-2m\left(n+\frac{1}{2}\right)\right)}
\]
is a {\it partial theta function}. Further, $\rho_0$ is the Weyl vector of $\gg_0$, $L$ is the root lattice of $\gg$, $W$ is its Weyl group and $W^\sharp$ the subgroup of $W$ corresponding to the {\it larger} subalgebra of $\gg_0$ and finally $I_\lambda$ is a subset of $\mathbb Z^{|\Delta_1^+|}$. More details on these objects are found in Section \ref{sec:ch}.
Summing over all representations corresponding to either $\gs\go(m)$ or $\gg\gl(m)$, one then obtains the character decomposition for $G=Sp(2n)$, and respectively for $G=GL(n)$. In particular, the character formula for $\cA(n)^{Sp(2n)}$ implies that it is {\it freely generated}; there are no nontrivial normally ordered relations among the generators $j^0, j^2, \dots, j^{2n-2}$. This yields an explicit classification of the irreducible, admissible $\cA(n)^{Sp(2n)}$-modules.

One of the important properties of rational vertex algebras is that characters of modules are the components of a vector-valued modular form for the modular group $SL(2, \mathbb Z)$ \cite{Zh}. In the non-rational case, the relation to modularity is unclear. However, there are examples of affine vertex superalgebras whose module characters are built out of mock modular forms \cite{KWI}. For the singlet vertex algebra $\cW(2, 2p-1)$ for $p\geq 2$, the module characters are, up to a prefactor of $\eta(q)^{-1}$, partial theta functions \cite{F}, and $\cW(2, 3)$ coincides with $\cA(1)^{GL(1)}$. The characters are then not modular, but their modular group action is still known \cite{Zw,CM}. They carry an infinite-dimensional modular group representation that is compatible with the expectations of vertex algebras and conformal field theory in the cases of a family of vertex superalgebras \cite{AC} and the singlet vertex algebras \cite{CM}.

\subsection{The Hilbert problem for $\cA(n)$}

A vertex algebra $\cV$ is called {\it strongly finitely generated} if there exists a finite set of generators such that the set of iterated Wick products of the generators and their derivatives spans $\cV$. Recall Hilbert's theorem that if a reductive group $G$ acts on a finite-dimensional complex vector space $V$, the ring $\cO(V)^G$ of invariant polynomial functions is finitely generated \cite{HI,HII}. The analogous question for vertex algebras is the following. 
Given a simple, strongly finitely generated vertex algebra $\cV$ and a reductive group $G$ of automorphisms of $\cV$, is $\cV^G$ strongly finitely generated? This problem was solved affirmatively when $\cV$ is the Heisenberg vertex algebra $\cH(n)$ \cite{LIII,LIV}, and when $\cV$ is the $\beta\gamma$-system $\cS(n)$ or free fermion algebra $\cF(n)$ \cite{LV}. The approach is the same in all these cases and is based on the following observations.
\begin{enumerate}
\item The Zhu algebra of  $\cV^{\text{Aut}(\cV)}$ is abelian, which implies that its irreducible, admissible modules are all highest-weight modules.
\item For any reductive $G$, $\cV^G$ has a strong generating set that lies in the direct sum of finitely many irreducible $\cV^{\text{Aut}(\cV)}$-modules.
\item The strong finite generation of $\cV^{\text{Aut}(\cV)}$ implies that these modules all have a certain finiteness property.
\end{enumerate}
In the case $\cV = \cA(n)$, these conditions are satisfied, so the same approach yields a solution to the Hilbert problem for $\cA(n)$; see Theorem \ref{sfg}. The method is essentially constructive and it reveals how $\cA(n)^G$ arises as an extension of $\cA(n)^{Sp(2n)}$ by irreducible modules. 

A {\it free field algebra} is any vertex algebra of the form $$\cH(n) \otimes \cF(m) \otimes \cS(r) \otimes \cA(s),\qquad m,n,r,s \geq 0,$$ where $\cV(0)$ is declared to be $\mathbb{C}$ for $\cV = \cH, \cS, \cF, \cA$. Many interesting vertex algebras have free field realizations, that is, embeddings into such algebras. In addition, many vertex algebras can be regarded as {\it deformations} of free field algebras. For example, if $\gg = \gg_0 \oplus \gg_1$ is a Lie superalgebra with a nondegenerate, supersymmetric bilinear form, the corresponding affine vertex superalgebra is a deformation of $\cH(n) \otimes \cA(m)$, where $n = \text{dim}(\gg_0)$ and $2m = \text{dim}(\gg_1)$. By combining Theorem \ref{sfg} with the results of \cite{LIII,LIV,LV}, one can establish the strong finite generation of a broad class of orbifolds of free field algebras, as well as their deformations. This has an application to the longstanding problem of describing {\it coset vertex algebras}, which appears in a separate paper \cite{CL}.

\section{Vertex algebras}
We shall assume that the reader is familiar with the basics of vertex algebra theory, which has been discussed from various points of view in the literature (see for example \cite{B,FLM,K,FBZ}). We will follow the formalism developed in \cite{LZ} and partly in \cite{LiI}, and we will use the notation of the previous paper \cite{LI} of the second author. In particular, by a {\it vertex algebra}, we mean a quantum operator algebra $\cA$ in which any two elements $a,b$ are local, meaning that $(z-w)^N[a(z), b(w)] = 0$ for some positive integer $N$. This is well known to be equivalent to the notion of vertex algebra in \cite{FLM}. The operators product expansion (OPE) formula is given by
$$a(z)b(w)\sim\sum_{n\geq 0}a(w)\circ_n b(w)\ (z-w)^{-n-1}.$$ Here $\sim$ means equal modulo terms which are regular at $z=w$, and $\circ_n$ denotes the $n^{\text{th}}$ circle product. A subset $S=\{a_i|\ i\in I\}$ of $\cA$ is said to {\it generate} $\cA$ if $\cA$ is spanned by words in the letters $a_i$, $\circ_n$, for $i\in I$ and $n\in\mathbb{Z}$. We say that $S$ {\it strongly generates} $\cA$ if every $\cA$ is spanned by words in the letters $a_i$, $\circ_n$ for $n<0$. Equivalently, $\cA$ is spanned by $$\{ :\partial^{k_1} a_{i_1}\cdots \partial^{k_m} a_{i_m}:| \ i_1,\dots,i_m \in I,\ k_1,\dots,k_m \geq 0\}.$$ We say that $S$ {\it freely generates} $\cA$ if there are no nontrivial normally ordered polynomial relations among the generators and their derivatives.

Our main example in this paper is the {\it symplectic fermion algebra} $\cA(n)$ of rank $n$, which is freely generated by odd elements $\{e^{i}, f^{i}|\ i=1,\dots, n\}$ satisfying
 \begin{equation} \begin{split}
e^{i} (z) f^{j}(w)&\sim \delta_{i,j} (z-w)^{-2},\qquad f^{j}(z) e^{i}(w)\sim - \delta_{i,j} (z-w)^{-2},\\
e^{i} (z) e^{j} (w)&\sim 0,\qquad\qquad\qquad\ \ \ f^{i} (z) f^{j} (w)\sim 0.
\end{split} \end{equation} We give $\cA(n)$ the conformal structure \begin{equation} \label{virasorosf} L^{\cA} =  - \sum_{i=1}^n  :e^i f^i: \end{equation} of central charge $-2n$, under which $e^i, f^i$ are primary of weight one. The full automorphism group of $\cA(n)$ is the symplectic group $Sp(2n)$. It acts linearly on the generators, which span a copy of the standard $Sp(2n)$-module $\mathbb{C}^{2n}$.

\subsection{Category $\mathcal{R}$}
Let $\cR$ be the category of vertex algebras $\cA$ equipped with a $\mathbb{Z}_{\geq 0}$-filtration
\begin{equation} \cA_{(0)}\subset\cA_{(1)}\subset\cA_{(2)}\subset \cdots,\qquad \cA = \bigcup_{k\geq 0}
\cA_{(k)}\end{equation} such that $\cA_{(0)} = \mathbb{C}$, and for all
$a\in \cA_{(k)}$, $b\in\cA_{(l)}$, we have
\begin{equation} \label{goodi} a\circ_n b\in  \bigg\{\begin{matrix}\cA_{(k+l)} & n<0 \\ \cA_{(k+l-1)} & 
n\geq 0 \end{matrix}\ . \end{equation}
Elements $a(z)\in\cA_{(d)}\setminus \cA_{(d-1)}$ are said to have degree $d$.

Filtrations on vertex algebras satisfying \eqref{goodi} were introduced in \cite{LiII}, and are known as {\it good increasing filtrations}. Setting $\cA_{(-1)} = \{0\}$, the associated graded object $\text{gr}(\cA) = \bigoplus_{k\geq 0}\cA_{(k)}/\cA_{(k-1)}$ is a $\mathbb{Z}_{\geq 0}$-graded associative, (super)commutative algebra, equipped with a derivation $\partial$ of degree zero. For each $r\geq 1$ we have the projection \begin{equation} \phi_r: \cA_{(r)} \ra \cA_{(r)}/\cA_{(r-1)}\subset \text{gr}(\cA).\end{equation} The assignment $\cA\mapsto \text{gr}(\cA)$ is a functor from $\cR$ to the category of $\partial$-rings, i.e., $\mathbb{Z}_{\geq 0}$-graded (super)commutative rings with a differential $\partial$ of degree zero. A $\partial$-ring $A$ is said to be generated by a subset $\{a_i|~i\in I\}$ if $\{\partial^k a_i|~i\in I, k\geq 0\}$ generates $A$ as a ring. The key feature of $\cR$ is the following reconstruction property \cite{LL}.

\begin{lemma}\label{reconlem}Let $\cA$ be a vertex algebra in $\cR$ and let $\{a_i|~i\in I\}$ be a set of generators for $\text{gr}(\cA)$ as a $\partial$-ring, where $a_i$ is homogeneous of degree $d_i$. If $a_i(z)\in\cA_{(d_i)}$ are elements such that $\phi_{d_i}(a_i(z)) = a_i$, then $\cA$ is strongly generated as a vertex algebra by $\{a_i(z)|~i\in I\}$.\end{lemma}

There is a similar reconstruction property for kernels of surjective morphisms \cite{LI}. Let $f:\cA\rightarrow \cB$ be a morphism in $\cR$ with kernel $\cJ$, such that $f$ maps each $\cA_{(k)}$ onto $\cB_{(k)}$. The kernel $J$ of the induced map $\text{gr}(f): \text{gr}(\cA)\rightarrow \text{gr}(\cB)$ is a homogeneous $\partial$-ideal (i.e., $\partial J \subset J$). A set $\{a_i|~i\in I\}$ such that $a_i$ is homogeneous of degree $d_i$ is said to generate $J$ as a $\partial$-ideal if $\{\partial^k a_i|~i\in I,~k\geq 0\}$ generates $J$ as an ideal.

\begin{lemma} \label{idealrecon} Let $\{a_i| i\in I\}$ be a generating set for $J$ as a $\partial$-ideal, where $a_i$ is homogeneous of degree $d_i$. Then there exist elements $a_i(z)\in \cA_{(d_i)}$ with $\phi_{d_i}(a_i(z)) = a_i$, such that $\{a_i(z)|~i\in I\}$ generates $\cJ$ as a vertex algebra ideal.\end{lemma}

We now define a good increasing filtration on the symplectic fermion algebra $\cA(n)$. First, $\cA(n)$ has a basis consisting of the normally ordered monomials
\begin{equation}\label{basisofs} :\partial^{I_1} e^{1}\cdots \partial^{I_n} e^{n}\partial^{J_1} f ^{1}\cdots\partial^{J_n} f^{n}: .\end{equation} In this notation, $I_k = (i^k_1,\dots, i^k_{r_k})$ and $J_k =(j^k_1,\dots, j^k_{s_k})$ are lists of integers satisfying $0\leq i^k_1< \cdots < i^k_{r_k}$ and  $0\leq j^k_1< \cdots <  j^k_{s_k}$, and 
$$\partial^{I_k} e^{k} = \ :\partial^{i^k_1} e^{k} \cdots \partial^{i^k_{r_k}} e^{k}:,\qquad \partial^{J_k} f^{k} = \ :\partial^{j^k_1} f^{k} \cdots \partial^{j^k_{s_k}} f^{k}:.$$ We have a $\mathbb{Z}_{\geq 0}$-grading \begin{equation}\label{grading} \cA(n) = \bigoplus_{d\geq 0} \cA(n)^{(d)},\end{equation} where $\cA(n)^{(d)}$ is spanned by monomials of the form \eqref{basisofs} of total degree $ d = \sum_{k=1}^n r_k + s_k$. Finally, we define the filtration $\cA(n)_{(d)} = \bigoplus_{i=0}^d \cA(n)^{(i)}$. This filtration satisfies \eqref{goodi}, and we have an isomorphism of $Sp(2n)$-modules
\begin{equation} \label{linassgrad} \cA(n) \cong  \text{gr}(\cA(n)),\end{equation}
and an isomorphism of graded supercommutative rings
\begin{equation} \label{genassgrad} \text{gr}(\cA(n))\cong \bigwedge \bigoplus_{k\geq 0} U_k.\end{equation} Here $U_k$ is the copy of the standard $Sp(2n)$-module $\mathbb{C}^{2n}$ with basis $\{e^{i}_k, f^{i}_k\}$. In this notation, $e^{i}_k$ and $f^{i}_k$ are the images of $\partial^k e^{i}(z)$ and $\partial^k f^{i}(z)$ in $\text{gr}(\cA(n))$. The $\partial$-ring structure on $\bigwedge \bigoplus_{k\geq 0} U_k$ is defined by $\partial e^i_k = e^i_{k+1}$ and $\partial f^i_k = f^i_{k+1}$. For any reductive group $G\subset Sp(2n)$, this filtration is $G$-invariant and is inherited by $\cA(n)^G$. We obtain a linear isomorphism $\cA(n)^G \cong \text{gr}(\cA(n)^G)$ and isomorphisms of $\partial$-rings \begin{equation} \label{currentringiso}  \text{gr}(\cA(n)^G)\cong \text{gr}(\cA(n))^G \cong \big(\bigwedge\bigoplus_{k\geq 0} U_k \big)^G. \end{equation} The weight grading on $\cA(n)$ is inherited by $\text{gr}(\cA(n))$ and \eqref{currentringiso} preserves weight as well as degree, where $\text{
wt}(e^i_k) =  \text{wt}(f^i_k) = k+1$.

\section{The structure of $\cA(n)^{Sp(2n)}$} 
Recall Weyl's first and second fundamental theorems of invariant theory for the standard representation of $Sp(2n)$ (Theorems 6.1.A and 6.1.B of \cite{We}).

\begin{thm} \label{weylfft} For $k\geq 0$, let $U_k$ be the copy of the standard $Sp(2n)$-module $\mathbb{C}^{2n}$ with symplectic basis $\{x_{i,k}, y_{i,k}| \ i=1,\dots,n\}$. Then $(\text{Sym} \bigoplus_{k\geq 0} U_k )^{Sp(2n)}$ is generated by the quadratics \begin{equation}\label{weylgenerators} q_{a,b} = \frac{1}{2}\sum_{i=1}^n \big( x_{i,a} y_{i,b} - x_{i,b} y_{i,a}\big),\qquad  0\leq a<b. \end{equation} For $a>b$, define $q_{a,b} = -q_{b,a}$, and let $\{Q_{a,b}|\ a,b\geq 0\}$ be commuting indeterminates satisfying $Q_{a,b} = -Q_{b,a}$ and no other algebraic relations. The kernel $I_n$ of the homomorphism \begin{equation}\label{weylquot} \mathbb{C}[Q_{a,b}]\ra (\text{Sym} \bigoplus_{k\geq 0} U_k)^{Sp(2n)},\qquad Q_{a,b}\mapsto q_{a,b},\end{equation} is generated by the degree $n+1$ Pfaffians $p_I$, which are indexed by lists $I = (i_0,\dots,i_{2n+1})$ of integers satisfying \begin{equation}\label{ijineq} 0\leq i_0<\cdots < i_{2n+1}.\end{equation} For $n=1$ and $I = (i_0, i_1, i_2, 
i_3)$, we have $$p_I = q_{i_0, i_1} q_{i_2, i_3} - q_{i_0, i_2} q_{i_1, i_3}+q_{i_0, i_3} q_{i_1, i_2},$$ and for $n>1$ they are defined inductively by \begin{equation} \label{pfaffinduction} p_I =  \sum_{r=1}^{2n+1} (-1)^{r+1} q_{i_0,i_r} p_{I_r},\end{equation} where $I_r = (i_1,\dots, \widehat{i_r},\dots, i_{2n+1})$ is obtained from $I$ by omitting $i_0$ and $i_r$. \end{thm}

There is an analogue of this theorem when the symmetric algebra $\text{Sym} \bigoplus_{k\geq 0} U_k$ is replaced by the exterior algebra $\bigwedge \bigoplus_{k\geq 0} U_k$. It is a special case of Sergeev's first and second fundamental theorems of invariant theory for $Osp(m,2n)$ (Theorem 1.3 of \cite{SI} and Theorem 4.5 of \cite{SII}). The generators of $(\bigwedge \bigoplus_{k\geq 0} U_k)^{Sp(2n)}$ are \begin{equation} \label{weylgeneratorsodd} q_{a,b} = \frac{1}{2}\sum_{i=1}^n \big( x_{i,a} y_{i,b} + x_{i,b} y_{i,a}\big),\qquad  0\leq a\leq b. \end{equation} For $a>b$, define $q_{a,b} = q_{b,a}$, and let $\{Q_{a,b}|\ a,b\geq 0\}$ be commuting indeterminates satisfying $Q_{a,b} = Q_{b,a}$ and no other algebraic relations. The kernel $I_n$ of the homomorphism \begin{equation}\label{weylquotodd} \mathbb{C}[Q_{a,b}]\ra (\bigwedge \bigoplus_{k\geq 0} U_k)^{Sp(2n)},\qquad Q_{a,b}\mapsto q_{a,b},\end{equation} is generated by elements $p_I$ of degree $n+1$ which are indexed by lists $I = (i_0,\dots,i_{2n+1})$ 
satisfying \begin{equation}\label{ijineqodd} 0\leq i_0\leq \cdots \leq i_{2n+1}.\end{equation} For $n=1$ and $I = (i_0, i_1, i_2, i_3)$, we have \begin{equation} \label{pfaffodd} p_I = q_{i_0, i_1} q_{i_2, i_3} + q_{i_0, i_2} q_{i_1, i_3}+q_{i_0, i_3} q_{i_1, i_2},\end{equation} and for $n>1$ they are defined inductively by \begin{equation} \label{pfaffinductionodd} p_I =  \sum_{r=1}^{2n+1}  q_{i_0,i_r} p_{I_r},\end{equation} where $I_r = (i_1,\dots, \widehat{i_r},\dots, i_{2n+1})$ is obtained from $I$ by omitting $i_0$ and $i_r$.

The generators $q_{a,b}$ of $R$ correspond to vertex operators \begin{equation}\label{newgenomega} \omega_{a,b} = \frac{1}{2}\sum_{i=1}^n \big(:\partial^a e^i \partial ^b f^i: + :\partial^b e^i \partial^a f^i:\big), \qquad 0\leq a\leq b,\end{equation} of $\cA(n)^{Sp(2n)}$, satisfying $\phi_2(\omega_{a,b}) = q_{a,b}$. By Lemma \ref{reconlem}, $\{\omega_{a,b}|~0\leq a \leq b\}$ strongly generates $\cA(n)^{Sp(2n)}$. In fact, there is a more economical strong generating set. For each $m\geq 0$, let $A_m$ denote the vector space spanned by $\{\omega_{a,b}|~ a+b = m\}$, which has weight $m+2$. We have $\text{dim}(A_{2m}) = m+1 = \text{dim}(A_{2m+1})$ for $m\geq 0$, so \begin{equation}\label{deca} \text{dim} \big(A_{2m} / \partial(A_{2m-1})\big) = 1,\qquad \text{dim} \big(A_{2m+1} / \partial(A_{2m})\big) = 0.\end{equation} For $m\geq 0$, define \begin{equation}\label{defofj} j^{2m} = \omega_{0,2m},\end{equation} which is clearly not a total derivative. We have \begin{equation}\label{decompofa} A_{2m} = \partial (A_{
2m-1})\oplus \bra j^{2m}\ket =  \partial^2 (A_{2m-2})\oplus \bra j^{2m}\ket ,\end{equation} where $\bra j^{2m}\ket$ is the linear span of $j^{2m}$. Similarly, \begin{equation}\label{decompofai} A_{2m+1} = \partial^2(A_{2m-1})\oplus \bra \partial j^{2m}\ket =  \partial^3 (A_{2m-2})\oplus \bra \partial j^{2m}\ket.\end{equation} Moreover, $\{\partial^{2i} j^{2m-2i}|~ 0\leq i\leq m\}$ and $\{\partial^{2i+1} j^{2m-2i}|\ 0\leq i\leq m\}$ are bases of $A_{2m}$ and $A_{2m+1}$, respectively, so each $\omega_{a,b}\in A_{2m}$ and $\omega_{c,d}\in A_{2m+1}$ can be expressed uniquely as \begin{equation}\label{lincomb} \omega_{a,b} =\sum_{i=0}^m \lambda_i \partial^{2i}j^{2m-2i},\qquad \omega_{c,d} =\sum_{i=0}^m \mu_i \partial^{2i+1}j^{2m-2i}\end{equation} for constants $\lambda_i,\mu_i$. Hence $\{j^{2m}|\ m\geq 0\}$ is also a strong generating set for $\cA(n)^{Sp(2n)}$.

\begin{thm} $\cA(n)^{Sp(2n)}$ is generated by $j^0$ and $j^2$ as a vertex algebra.
\end{thm}

\begin{proof} It suffices to show that each $j^{2k}$ can be generated by these elements. This follows from the calculation $$j^2 \circ_1 j^{2k} = -(2k+4) j^{2k+2}+ \partial^2 \omega,$$ where $\omega$ is a linear combination of $\partial^{2i}j^{2k-2i}$ for $i=0,\dots, k$.
\end{proof}

Consider the category of vertex algebras with generators $\{J^{2m}|\ m\geq 0\}$, which satisfy the same OPE relations as the generators $\{j^{2m}|~m\geq 0\}$ of $\cA(n)^{Sp(2n)}$. Since the vector space with basis $\{1\}\cup \{\partial^lj^{2m}|\ l,m\geq 0\}$ is closed under $\circ_n$ for all $n\geq 0$, it forms a Lie conformal algebra. By Theorem 7.12 of \cite{BK}, this category contains a universal object $\cM_n$, which is {\it freely} generated by $\{J^{2m}|\ m\geq 0\}$. Then $\cA(n)^{Sp(2n)}$ is a quotient of $\cM_n$ by an ideal $\cI_n$, and since $\cA(n)^{Sp(2n)}$ is a simple vertex algebra, $\cI_n$ is maximal. Let $$\pi_n: \cM_n \rightarrow \cA(n)^{Sp(2n)},\qquad J^{2m}\mapsto j^{2m}$$ denote the quotient map. Using (\ref{lincomb}), which holds in $\cA(n)^{Sp(2n)}$ for all $n$, we can define an alternative strong generating set $\{\Omega_{a,b}| ~0\leq a\leq b\}$ for $\cM_n$ by the same formula: for $a+b = 2m$ and $c+d = 2m+1$, $$\Omega_{a,b} =\sum_{i=0}^m \lambda_i \partial^{2i}J^{2m-2i},\qquad 
 \Omega_{c,d} =\sum_{i=0}^m \mu_i \partial^{2i+1}J^{2m-2i}.$$ Clearly $\pi_n(\Omega_{a,b}) = \omega_{a,b}$. We shall use the same notation $A_m$ to denote the span of $\{\Omega_{a,b}|\ a+b = m\}$, when no confusion can arise. Note that $\cM_{n}$ has a good increasing filtration in which $(\cM_n)_{(2k)}$ is spanned by iterated Wick products of the generators $J^{2m}$ and their derivatives, of length at most $k$, and $(\cM_n)_{(2k+1)} = (\cM_n)_{(2k)}$. Equipped with this filtration, $\cM_n$ lies in the category $\cR$, and $\pi_n$ is a morphism in $\cR$.

\subsection{The structure of the ideal $\cI_{n}$}
Under the identifications $$\text{gr}(\cM_{n})\cong \mathbb{C}[Q_{a,b}],\qquad \text{gr}(\cA(n)^{Sp(2n)})\cong ( \bigwedge \bigoplus_{k\geq 0} U_k)^{Sp(2n)}\cong \mathbb{C}[q_{a,b}]/I_n,$$ $\text{gr}(\pi_{n})$ is just the quotient map \eqref{weylquot}. 

\begin{lemma} \label{ddef} For each $I = (i_0,i_1,\dots, i_{2n+1})$, there exists a unique element \begin{equation} \label{ddefeq} P_{I}\in (\cM_{n})_{(2n+2)}\cap \cI_{n}\end{equation} of weight $2n+2 +\sum_{a=0}^{2n+1} i_a$, satisfying \begin{equation}\label{uniquedij} \phi_{2n+2}(P_{I}) = p_{I}.\end{equation} These elements generate $\cI_{n}$ as a vertex algebra ideal. \end{lemma} 

\begin{proof}
Clearly $\pi_{n}$ maps each $(\cM_{n})_{(k)}$ onto $(\cA(n)^{Sp(2n)})_{(k)}$, so the hypotheses of Lemma \ref{idealrecon} are satisfied. Since $I_{n} = \text{Ker} (\text{gr}(\pi_{n}))$ is generated by $\{p_{I}\}$, we can apply Lemma \ref{idealrecon} to find $P_{I}\in (\cM_{n})_{(2n+2)}\cap \cI_{n}$ satisfying $\phi_{2n+2}(P_{I}) = p_{I}$, such that $\{P_{I}\}$ generates $\cI_{n}$. If $P'_{I}$ also satisfies \eqref{uniquedij}, we would have $P_{I} - P'_{I}\in (\cM_{n})_{(2n)} \cap \cI_{n}$. Since there are no relations in $\cA(n)^{Sp(2n)}$ of degree less than $2n+2$, we have $P_{I} - P'_{I}=0$. \end{proof}

Let $\bra P_I \ket$ denote the vector space with basis $\{P_I\}$ where $I$ satisfies \eqref{ijineq}. We have $$\bra P_I \ket = (\cM_{n})_{(2n+2)}\cap \cI_{n},$$ and clearly $\bra P_I \ket$ is a module over the Lie algebra $\cP$ generated by $\{J^{2m}(k) |~ m,k \geq 0\}$. It is convenient to work with a different generating set for $\cP$, namely $$\{\Omega_{a,b} (a+b+1-w)|~ 0\leq a\leq b,\  a+b+1-w\geq 0\}.$$ Note that $\Omega_{a,b}(a+b+1-w)$ is homogeneous of weight $w$, and $\cP$ acts on $\text{gr}(\cM_{n})$ by derivations of degree zero. This action is independent of $n$, and is specified by the action on the generators $\Omega_{c,d}$. We compute \begin{equation} \label{actionp} \Omega_{a,b}(a+b+1-w) (\Omega_{c,d}) = \lambda_{a,b,w,c} (\Omega_{c+w,d}) + \lambda_{a,b,w,d} ( \Omega_{c,d+w}),\end{equation} where \begin{equation} \label{actionlambdap}\lambda_{a,b,w,c}  =  
\bigg\{ \begin{matrix}  (-1)^{a+1} \frac{(a+c+1)!}{(c-b+w)!} + (-1)^{b+1} \frac{(b+c+1)!}{(c-a+w)!} & c-b+w \geq 0 \cr & \cr 0 & c-b+w <0 \end{matrix}.\end{equation}
The action of $\cP$ on $\bra P_I \ket$ is by {\it weighted derivation} in the following sense. Given $I = (i_0,\dots,i_{2n+1})$ and $p= \Omega_{a,b}(a+b+1-w)\in \cP$, we have \begin{equation} \label{weightederivation} p(P_{I}) = \sum_{r=0}^{2n+1} \lambda_r P_{I^r},\end{equation} where $I^r = (i_0,\dots, i_{r-1}, i_r + w,i_{r+1},\dots, i_{2n+1})$ and $\lambda_r = \lambda_{a,b,w,i_r}$.

For each $n\geq 1$, there is a distinguished element $P_0\in \mathcal{I}_n$, defined by $$P_0 = P_{I},\qquad I = (0,0,\dots, 0).$$ It is the unique element of $\cI_{n}$ of minimal weight $2n+2$, and is a singular vector in $\cM_{n}$. In fact, $P_0$ generates $\mathcal{I}_{n}$ as a vertex algebra ideal. The proof is similar to the proof of Theorem 7.2 of \cite{LV}, and involves using \eqref{actionp}-\eqref{weightederivation} to show that $\bra P_I \ket$ is generated by $P_0$ as a module over $\mathcal{P}$. This is sufficient because $\bra P_I \ket$ generates $\cI_{n}$ as a vertex algebra ideal.

\subsection{Normal ordering and quantum corrections}

Given a homogeneous $p\in \text{gr}(\cM_{n})\cong \mathbb{C}[Q_{a,b}]$ of degree $k$ in the $Q_{a,b}$, a {\it normal ordering} of $p$ will be a choice of normally ordered polynomial $P\in (\cM_{n})_{(2k)}$, obtained by replacing $Q_{a,b}$ by $\Omega_{a,b}$, and by replacing ordinary products with iterated Wick products. For any choice we have $\phi_{2k}(P) = p$. For the rest of this section, $P^{2k}$, $E^{2k}$, $F^{2k}$, etc., will denote elements of $(\cM_{n})_{(2k)}$ which are homogeneous, normally ordered polynomials of degree $k$ in the $\Omega_{a,b}$.

Let $P_{I}^{2n+2}\in (\cM_{n})_{(2n+2)}$ be some normal ordering of $p_I$, so that $\phi_{2n+2}(P_I^{2n+2}) = p_I$. Then $$\pi_{n}(P_I^{2n+2}) \in (\cA(n)^{Sp(2n)})_{(2n)},$$ and $\phi_{2n}(\pi_{n}(P_I^{2n+2})) \in \text{gr}(\cA(n)^{Sp(2n)})$ can be expressed uniquely as a polynomial of degree $n$ in the variables $q_{a,b}$. Choose some normal ordering of the corresponding polynomial in the variables $\Omega_{a,b}$, and call this element $-P^{2n}_{I}$. Then $P^{2n+2}_{I} + P^{2n}_{I}$ satisfies $$\phi_{2n+2}(P_I^{2n+2} + P^{2n}_{I}) = p_I,\qquad \pi_{n}(P^{2n+2}_{I} + P^{2n}_{I})\in (\cA(n)^{Sp(2n)})_{(2n-2)}.$$ Continuing this process, we obtain an element $\sum_{k=1}^{n+1} P^{2k}_{I}$ in the kernel of $\pi_{n}$, such that $\phi_{2n+2}(\sum_{k=1}^{n+1} P^{2k}_{I}) = p_I$. By Lemma \ref{ddef}, must have 
\begin{equation}\label{decompofd} P_{I} = \sum_{k=1}^{n+1}P^{2k}_{I}.\end{equation} The term $P^2_{I}$ lies in the space $A_m$ spanned by $\{\Omega_{a,b}|~a+b=m\}$, for $m = 2n+ \sum_{a=0}^{2n+1} i_a$. By \eqref{decompofa}, for all even integers $m\geq 1$ we have a projection $$\text{pr}_{m}: A_{m}\ra \bra J^{m}\ket.$$ For all $I = (i_0,i_1,\dots, i_{2n+1})$ such that $m = 2n+ \sum_{a=0}^{2n+1} i_a$ is even, define the {\it remainder} \begin{equation}\label{defofrij} R_{I} = \text{pr}_m(P^2_{I}).\end{equation}

\begin{lemma} \label{uniquenessofr} Fix $P_{I}\in\cI_{n}$ with $I = (i_0,i_1, \dots, i_{2n+1})$ and $m = 2n+ \sum_{a=0}^{2n+1} i_a$ even. Suppose that $P_{I} = \sum_{k=1}^{n+1} P^{2k}_{I}$ and $P_{I} = \sum_{k=1}^{n+1} \tilde{P}^{2k}_{I}$ are two different decompositions of $P_{I}$ of the form (\ref{decompofd}). Then $$P^2_{I} - \tilde{P}^2_{I} \in \partial^2 (A_{m-2}).$$ In particular, $R_{I}$ is independent of the choice of decomposition of $P_{I}$.\end{lemma}

\begin{proof} The argument is the same as the proof of Lemma 4.7 and Corollary 4.8 of \cite{LI}. \end{proof}

\begin{lemma}\label{lemnonzero} Let $R_0$ denote the remainder of the element $P_0$. The condition $R_0 \neq 0$ is equivalent to the existence of a decoupling relation in $\cA(n)^{Sp(2n)}$ of the form \begin{equation}\label{maindecoupling} j^{2n} = Q(j^0,j^2,\dots, j^{2n-2}),\end{equation} where $Q$ is a normally ordered polynomial in $j^0, j^2, \dots, j^{2n-2}$ and their derivatives. \end{lemma}

\begin{proof} 
This is the same as the proof of Lemma 4.9 of \cite{LI}. \end{proof}

\begin{lemma} \label{lemnonzeroii} Suppose that $R_0 \neq 0$. Then for all $m\geq n$, there exists a decoupling relation \begin{equation}\label{hdrelm} j^{2m} = Q_m(j^0, j^2, \dots, j^{2n-2}).\end{equation} Here $Q_m$ is a normally ordered polynomial in $j^0, j^2, \dots, j^{2n-2}$, and their derivatives. \end{lemma}

\begin{proof} This is the same as the proof of Lemma 8.3 of \cite{LV}. \end{proof}

\subsection{A recursive formula for $R_I$}
\label{recursionanalysis}
In this section we find a recursive formula for $R_I$ for any $I = (i_0,i_1,\dots, i_{2n+1})$ such that $\text{wt}(P_I) = 2n+2+ \sum_{a=0}^{2n+1} i_a$ is even. It will be clear from our formula that $R_0 \neq 0$. We introduce the notation \begin{equation} \label{remcoeff} R_{I} = R_n(I) J^m,\ \ \ \ \ \ \ m = 2n+ \sum_{a=0}^{2n+1} i_a,\end{equation} so that $R_n(I)$ denotes the coefficient of $J^m$ in $\text{pr}_m(P^2_{I})$. For $n=1$ and $I = (i_0, i_1, i_2, i_3)$ the following formula is easy to obtain using the fact that $\text{pr}_m(\Omega_{a,b}) =(-1)^m J^{m}$ for $m=a+b$.

\begin{equation}\label{startrecur}
\begin{split}
R_1(I) & = -\frac{1}{8} \bigg(  \frac{(-1)^{i_0 + i_2} +  (-1)^{i_0 + i_3} + (-1)^{i_1 + i_2}  +  (-1)^{i_1 + i_3}}{2 + i_0 + i_1}\\
&+ \frac{(-1)^{i_0 + i_1} +  (-1)^{i_0 + i_3} + (-1)^{i_1 + i_2}  +  (-1)^{i_2 + i_3}}{2 + i_0 + i_2} \\
&+ \frac{(-1)^{i_0+ i_1} +  (-1)^{i_0 + i_2} +  (-1)^{i_1 + i_3} +  (-1)^{i_2 + i_3}}{2 + i_1 + i_2} \\
&+ \frac{(-1)^{i_0 + i_1} + (-1)^{i_0+ i_2} +  (-1)^{i_1 + i_3} +  (-1)^{i_2 + i_3}}{2 + i_0 + i_3} \\
&+ \frac{(-1)^{i_0 + i_1} + (-1)^{i_0+ i_3}+ (-1)^{i_1 + i_2}  +  (-1)^{i_2 + i_3} }{2 + i_1 + i_3} \\
&+  \frac{ (-1)^{i_0 + i_2}+(-1)^{i_0 + i_3} + (-1)^{i_1 + i_2} + (-1)^{i_1+ i_3}}{2 + i_2 + i_3} \bigg).
 \end{split} \end{equation}

Assume that $R_{n-1}(J)$ has been defined for all $J$. Recall first that $\cA(n)$ is a graded algebra with $\mathbb{Z}_{\geq 0}$ grading (\ref{grading}), which specifies a linear isomorphism $$\cA(n)\cong \bigwedge \bigoplus_{k\geq 0} U_k,\ \ \ \ \ \ U_k \cong \mathbb{C}^{2n}.$$ Since $\cA(n)^{Sp(2n)}$ is a graded subalgebra of $\cA(n)$, we obtain an isomorphism of graded vector spaces \begin{equation}\label{linisomor} i_{n}: \cA(n)^{Sp(2n)} \rightarrow (\bigwedge \bigoplus_{k\geq 0} U_k)^{Sp(2n)}.\end{equation} Let $p\in (\bigwedge \bigoplus_{k\geq 0} U_k )^{Sp(2n)}$ be homogeneous of degree $2d$, and let $$f = (i_{n})^{-1}(p)\in (\cA(n)^{Sp(2n)})^{(2d)}$$ be the corresponding homogeneous element. Let $F\in (\cM_{n})_{(2d)}$ be an element satisfying $\pi_{n}(F) = f$, where $\pi_{n}: \cM_{n}\rightarrow \cA(n)^{Sp(2n)}$ is the projection. We can write $F = \sum_{k=1}^{d} F^{2k}$, where $F^{2k}$ is a normally ordered polynomial of degree $k$ in the $\Omega_{a,b}$.

Next, consider the rank $n+1$ symplectic fermion algebra $\cA(n+1)$, and let $$\tilde{q}_{a,b}\in (\bigwedge \bigoplus_{k\geq 0} \tilde{U}_k)^{Sp(2n+2)}\cong \text{gr}(\cA(n+1))^{Sp(2n+2)} \cong \text{gr}(\cA(n+1)^{Sp(2n+2)})$$ be the generator given by (\ref{weylgenerators}), where $\tilde{U}_k \cong \mathbb{C}^{2n+2}$. Let $\tilde{p}$ be the polynomial of degree $2d$ obtained from $p$ by replacing each $q_{a,b}$ with $\tilde{q}_{a,b}$, and let $$\tilde{f} = (i_{n+1})^{-1} (\tilde{p}) \in (\cA(n+1))^{Sp(2n+2)})^{(2d)}$$ be the corresponding homogeneous element. Finally, let $\tilde{F}^{2k}\in \cM_{n+1}$ be the element obtained from $F^{2k}$ by replacing each $\Omega_{a,b}$ with the corresponding generator $\tilde{\Omega}_{a,b}\in \cM_{n+1}$, and let $\tilde{F} = \sum_{i=1}^d \tilde{F}^{2k}$. 

\begin{lemma} \label{corhomo} Fix $n\geq 1$, and let $P_{I}$ be an element of $\cI_{n}$ given by Lemma \ref{ddef}. There exists a decomposition $P_{I} = \sum_{k=1}^{n+1} P^{2k}_{I}$ of the form \eqref{decompofd} such that the corresponding element $$\tilde{P}_{I} = \sum_{k=1}^{n+1} \tilde{P}^{2k}_{I} \in \cM_{n+1}$$ has the property that $\pi_{n+1}(\tilde{P}_{I})$ lies in the homogeneous subspace $(\cA(n+1)^{Sp(2n+2)})^{(2n+2)}$.
\end{lemma}

\begin{proof} The argument is the same as the proof of Corollary 4.14 of \cite{LI}, and is omitted.
\end{proof}

Recall that $p_I$ has an expansion $p_I =  \sum_{r=1}^{2n+1} q_{i_0,i_r} p_{I_r}$, where $I_r = (i_1,\dots, \widehat{i_r},\dots, i_{2n+1})$ is obtained from $I$ by omitting $i_0$ and $i_r$. Let $P_{I_r} \in \cM_{n-1}$ be the element corresponding to $p_{I_r}$. By Lemma \ref{corhomo}, there exists a decomposition $$P_{I_r} = \sum_{i=1}^n P^{2i}_{I_r}$$ such that the corresponding element $\tilde{P}_{I_r} = \sum_{i=1}^n \tilde{P}^{2i}_{I_r} \in \cM_{n}$ has the property that $\pi_{n}(\tilde{P}_{I_r})$ lies in $(\cA(n)^{Sp(2n)})^{(2n)}$. We have \begin{equation}\label{usefuli} \sum_{r=1}^{2n+1} :\Omega_{i_0, i_r} \tilde{P}_{I_r}:\  = \sum_{r=1}^{2n+1} \sum_{i=1}^n :\Omega_{i_0,i_r} \tilde{P}^{2i}_{I_r}:.\end{equation} 

The right hand side of \eqref{usefuli} consists of normally ordered monomials of degree at least $2$ in the generators $\Omega_{a,b}$, and hence contributes nothing to $R_n(I)$. Since $\pi_{n}(\tilde{P}_{I_r})$ is homogeneous of degree $2n$, $\pi_{n}(:\Omega_{i_0,i_r} \tilde{P}_{I_r}:)$ consists of a piece of degree $2n+2$ and a piece of degree $2n$ coming from all double contractions of $\Omega_{i_0,i_r}$ with terms in $\tilde{P}_{I_r}$, which lower the degree by two. The component of $$\pi_{n}\bigg(\sum_{r=1}^{2n+1} :\Omega_{i_0, i_r} \tilde{P}_{I_r}:\bigg)\in \cA(n)^{Sp(2n)}$$ in degree $2n+2$ must cancel since this sum corresponds to $p_I$, which is a relation among the variables $q_{a,b}$. The component of $:\Omega_{i_0,i_r}\tilde{P}_{I_r}:$ in degree $2n$ is
\begin{equation} \label{crunch} S_r= \frac{1}{2} \bigg( (-1)^{i_0+1} \sum_{a} \frac{\tilde{P}_{I_{r,a}}}{i_0+i_{a} +2} + (-1)^{i_r+1} \sum_{a}  \frac{\tilde{P}_{I_{r,a}}}{i_r+i_{a}+2}\bigg)\end{equation}
In this notation, for $a \in \{i_0,\dots, i_{2n+1}\} \setminus \{i_0, i_r\}$, $I_{r,a}$ is obtained from $I_r = (i_1,\dots, \widehat{i_r},\dots, i_{2n+1})$ by replacing $i_a$ with $i_a+i_0+i_r+2$. It follows that \begin{equation} \label{usefulii} \pi_{n}\bigg(\sum_{r=1}^{2n+1}  :\Omega_{i_0, i_r} \tilde{P}_{I_r}:\bigg)= \pi_{n}\bigg(\sum_{r=1}^{2n+1}  S_r \bigg).\end{equation} Combining \eqref{usefuli} and \eqref{usefulii}, we can regard $$\sum_{r=1}^{2n+1} \sum_{i=1}^n  :\Omega_{i_0,i_r} \tilde{P}^{2i}_{I_r}: - \sum_{r=0}^n S_r$$ as a decomposition of $P_{I}$ of the form $P_{I} = \sum_{k=1}^{n+1} P^{2k}_{I}$ where the leading term $P^{2n+2}_{I} = \sum_{r=0}^{2n+1}  :\Omega_{i_0, i_r} \tilde{P}^{2n}_{I_r}:$. Therefore $R_n(I)$ is the negative of the sum of the terms $R_{n-1}(J)$ corresponding to each $\tilde{P}_{J}$ appearing in $\sum_{r=0}^{2n+1}  S_r$, so we obtain the following result.

\begin{thm} \label{recformula} $R_n(I)$ satisfies the recursive formula
\begin{equation} \label{recursion}R_n(I) = -\frac{1}{2} \sum_{r=1}^{2n+1} \bigg( (-1)^{i_0+1} \sum_{a} \frac{R_{n-1}(I_{r,a})}{i_0+i_{a} +2} + (-1)^{i_r+1} \sum_{a}  \frac{R_{n-1}(I_{r,a})}{i_r+i_{a}+2}\bigg) .\end{equation} \end{thm}

Now suppose that all the entries $i_0,\dots, i_{2n+1}$ appearing in $I$ are even. Clearly each $I_{r,a}$ appearing in \eqref{recursion} consists only of even entries as well. In the case $n=1$ and $I = (i_0,i_1, i_2, i_3)$, \eqref{startrecur} reduces to
$$R_1(I) = -\frac{1}{2} \bigg(\frac{1}{2 + i_0 + i_1} + \frac{1}{2 + i_0 + i_2} + \frac{1}{2 + i_1 + i_2} +\frac{1}{2 + i_0 + i_3} + \frac{1}{2 + i_1 + i_3}  + \frac{1}{2 + i_2 + i_3} \bigg),$$ so in particular $R_1(I) \neq 0$. By induction on $n$, it is immediate from \eqref{recursion} that $R_n(I) \neq 0$ whenever $I$ has even entries.

\begin{thm} \label{maincor}For all $n\geq 1$, $\cA(n)^{Sp(2n)}$ has a minimal strong generating set $\{j^0, j^{2}, \dots, j^{2n-2}\}$, and is therefore a $\cW$-algebra of type $\cW(2,4,\dots, 2n)$. \end{thm}

\begin{proof} For $I = (0,0,\dots, 0)$, we have $R_n(I) \neq 0$, so $R_0 = R_n(I) J^{2n} \neq 0$. The claim then follows from Lemma \ref{lemnonzeroii}. \end{proof}

Using Theorem \ref{maincor}, it is now straightforward to describe $\cA(mn)^{Sp(2n)}$ for all $m,n\geq 1$. Denote the generators of $\cA(mn)$ by $e^{i,j}, f^{i,j}$ for $i=1,\dots,n$ and $j=1,\dots, m$, which satisfy
$$e^{i,j}(z) f^{k,l}(w) \sim \delta_{i,k} \delta_{j,l} (z-w)^{-2}.$$ Under the action of $Sp(2n)$, $\{e^{i,j}, f^{i,j}|\ i=1,\dots, n\}$ spans a copy of the standard $Sp(2n)$-module $\mathbb{C}^{2n}$ for each $j=1,\dots, m$. Define
$$\omega^{j,k}_{a,b} = \frac{1}{2}\sum_{i=1}^n \big(:\partial^a e^{i,j} \partial ^b f^{i,k}: + :\partial^b e^{i,k} \partial^a f^{i,j}:\big).$$
By Theorem \ref{weylfft}, $\{\omega^{j,k}_{a,b}|\ 1\leq j,k \leq m,\ a,b\geq 0\}$ strongly generates $\cA(mn)^{Sp(2n)}$. In fact, $$\{\omega^{j,j}_{0,2k}|\ k\geq 0,\ j=1,\dots, m\} \bigcup \{\omega^{j,k}_{0,l}\ 1\leq j<k\leq m,\ l\geq 0\}$$ also strongly generates $\cA(mn)^{Sp(2n)}$. This is clear because the corresponding elements of $\text{gr}(\cA(mn)^{Sp(2n)})$ generate $\text{gr}(\cA(mn)^{Sp(2n)})$ as a $\partial$-ring. 

\begin{thm} \label{mn} $\cA(mn)^{Sp(2n)}$ has a minimal strong generating set \begin{equation} \label{fullsetofgen} \{\omega^{j,j}_{0,2r}|\ 1\leq j \leq m,\ 0\leq r \leq n-1\} \bigcup \{ \omega^{j,k}_{0,s}|\ 1\leq j<k \leq m,\ 0\leq s \leq  2n-1\}.\end{equation}
\end{thm}

\begin{proof} By Theorem \ref{maincor}, we have decoupling relations \begin{equation}\label{eachj} \omega^{j,j}_{0,2r} = Q_r(\omega^{j,j}_{0,0},\dots, \omega^{j,j}_{0,2n-2}),\end{equation} for all $r\geq n$ and $j=1,\dots, m$. We now construct decoupling relations expressing each $\omega^{j,k}_{0,a}$ as a normally polynomial in the generators \eqref{fullsetofgen} and their derivatives, for $j<k$ and $a\geq 2n$. We need the following calculations.
 $$\omega^{j,k}_{0,0} \circ_1 \omega^{j,j}_{0,2r} = -(r+1) \omega^{j,k}_{0,2r} + \partial \omega,\qquad \omega^{j,k}_{0,0} \circ_1 \partial \omega^{j,j}_{0,2r+1} = - \omega^{j,k}_{0,2r+1} + \partial \nu.$$ Here $\omega$ is a linear combination of $\partial^{2r-t}\omega^{j,k}_{0,t}$ for $t=0,\dots, 2r-1$, and $\nu$ is a linear combination of $\partial^{2r-t+1}\omega^{j,k}_{0,t}$ for $t=0,\dots, 2r$. Therefore applying the operator $\omega^{j,k}_{0,0} \circ_1$ to \eqref{eachj} yields a relation
 $$\omega^{j,k}_{0,2r} = \tilde{Q}_r(\omega^{j,j}_{0,0},\omega^{j,j}_{0,2},\dots \omega^{j,j}_{0,2n-2},\omega^{j,k}_{0,0}, \omega^{j,k}_{0,1},\dots, \omega^{j,k}_{0,2n-1}),$$ for all $r\geq n$. Similarly, applying $\omega^{j,k}_{0,0} \circ_1$ to the derivative of \eqref{eachj} yields a relation $$\omega^{j,k}_{0,2r+1} = \bar{Q}_r(\omega^{j,j}_{0,0},\omega^{j,j}_{0,2},\dots \omega^{j,j}_{0,2n-2},\omega^{j,k}_{0,0}, \omega^{j,k}_{0,1},\dots, \omega^{j,k}_{0,2n-1}),$$ for all $r\geq n$. This shows that \eqref{fullsetofgen} is a strong generating set for $\cA(mn)^{Sp(2n)}$. It is {\it minimal} because there are no normally ordered relations of weight less than $2n+2$.
\end{proof}

\section{The structure of $\cA(n)^{GL(n)}$}
The subgroup $GL(n)\subset Sp(2n)$ acts on $\cA(n)$ such that the generators $\{e^i\}$ and $\{f^i\}$ of $\cA(n)$ span copies of the standard $GL(n)$-modules $\mathbb{C}^n$ and $(\mathbb{C}^*)^n$, respectively. In this section, we use a similar approach to find a minimal strong generating set for $\cA(n)^{GL(n)}$. First, we have isomorphisms $$\text{gr}(\cA(n)^{GL(n)}) \cong \text{gr}(\cA(n))^{GL(n)} \cong  R := \big( \bigwedge \bigoplus_{j\geq 0} (V_j \oplus V^*_j)\big)^{GL(n)},$$ where $V_j \cong \mathbb{C}^n$ and $V^*_j \cong (\mathbb{C}^n)^*$ as $GL(n)$-modules. By an odd analogue of Weyl's first and second fundamental theorems of invariant theory for the standard representation of $GL(n)$, $R$ is generated by the quadratics $p_{a,b} = \sum_{i=1}^n x_a^i y_b^i$ where $\{x^i_a\}$ is a basis for $V_a$ and $\{y^i_b\}$ is the dual basis for $V^*_b$. The ideal of relations is generated by elements $d_{I,J}$ of degree $n+1$, which are indexed by lists $I = (i_0,i_1,\dots, i_n)$ and $J = (j_0, j_1, \dots, j_n)$ of integers satisfying $0\leq i_0 \leq \cdots \leq i_n$ and $0\leq j_0 \leq \cdots \leq j_n$, which are analogous to determinants but without the signs. (This is a special case of Theorems 2.1 and 2.2 of \cite{SII}). For $n=1$, $$d_{I,J} = p_{i_0, j_0} p_{i_1, j_1}+ p_{i_1, j_0} p_{i_0, j_1},$$ and for $n>1$, $d_{I,J}$ is defined inductively by $$d_{I,J} = \sum_{r=0}^n p_{i_r, j_0} d_{I_r, J'},$$ where $I_r = (i_0,\dots, \hat{i_r} \dots, i_n)$ is obtained from $I$ by omitting $i_r$, and $J' = (j_1,\dots, j_n)$ is obtained from $J$ by omitting $j_0$.

The generators $p_{a,b}$ correspond to strong generators 
$$\gamma_{a,b} = \sum_{i=1}^n :\partial^a e^i \partial^b f^i:,\qquad a,b \geq 0,$$ for $\cA(n)^{GL(n)}$, where $\text{wt}(\gamma_{a,b}) = a+b+2$. Let $A_m$ be spanned by $\{\gamma_{a,b}|\ a+b = m\}$. Then $\text{dim}(A_m) = m+1$ and $\text{dim}(A_m / \partial A_{m-1}) = 1$, so $$A_m = \partial A_{m-1} \oplus \langle h^m \rangle,\qquad h^m = \gamma_{0,m}.$$ Since $\{\partial^a h^{m-a}|\ a=0,\dots, m\}$ and $\{\gamma_{a,m-a}|\ a=0,\dots, m\}$ are both bases for $A_m$, $\{h^k |\ k\geq 0\}$ strongly generates $\cA(n)^{GL(n)}$.

\begin{lemma} $\cA(n)^{GL(n)}$ is generated as a vertex algebra by $h^0$ and $h^1$.
\end{lemma}
\begin{proof} This is immediate from the following calculation.
$$h^1 \circ_1 h^k = -(k+3) h^{k+1}+2 \partial h^k,\qquad k\geq 1.$$
\end{proof}

There exists a vertex algebra $\cN_n$ which is freely generated by $\{H^k|\ k\geq 0\}$ with same OPE relations as $\{h^k|\ k\geq 0\}$, such that $\cA(n)^{GL(n)}$ is a quotient of $\cN_n$ by an ideal $\cJ_n$ under a map $$\pi_n: \cN_n\ra \cA(n)^{GL(n)},\qquad H^k \ra h^k.$$ We have an alternative strong generating set $\{\Gamma_{a,b}|\ a,b\geq 0\}$ for $\cN_n$ satisfying $\pi_n(\Gamma_{a,b}) = \gamma_{a,b}$. There is a good increasing filtration on $\cN_n$ where $(\cN_n)_{(2k)}$ is spanned by iterated Wick products of $H^m$ and their derivatives of length at most $k$, and $(\cN_n)_{(2k+1)} = (\cN_n)_{(2k)}$, and $\pi_n$ preserves filtrations.

\begin{lemma} For each $I$ and $J$ as above, there exists a unique element $D_{I,J} \in (\cN_n)_{(2n+2)} \cap \cJ_n$ of weight $2n+2 + \sum_{a=0}^{n+1}(i_a + j_a)$, satisfying $\phi_{2n+2}(D_{I,J}) = d_{I,J}$. These elements generate $\cJ_n$ as a vertex algebra ideal.
\end{lemma}

Each $D_{I,J}$ can be written in the form \begin{equation} \label{decompdij} D_{I,J} = \sum_{k=1}^{n+1} D^{2k}_{I,J},\end{equation} where $D^{2k}_{I,J}$ is a normally ordered polynomial of degree $k$ in the generators $\Gamma_{a,b}$. The term $D^2_{I,J}$ lies in the space $A_m$ for $m= 2n+ \sum_{a=0}^n (i_a+j_a)$. We have the projection $\text{pr}_m: A_m \ra \langle H^m \rangle $, and we define the remainder $$R_{I,J} = \text{pr}_m(D^2_{I,J}).$$ It is independent of the choice of decomposition \eqref{decompdij}.

The element of $\cJ_n$ of minimal weight corresponds to $I = (0,\dots, 0) = J$, and has weight $2n+2$. We denote this element by $D_0$ and we denote its remainder by $R_0$. The condition $R_0 \neq 0$ is equivalent to the existence of a decoupling relation
\begin{equation} \label{lowestrel} h^{2n} = P(h^0, h^1,\dots, h^{2n-1}),\end{equation} where $P$ is a normally ordered polynomial in $h^0, h^1,\dots, h^{2n-1}$ and their derivatives. From this relation it is easy to construct decoupling relations $h^m = P_m(h^0, h^1, \dots, h^{2n-1})$ for all $m>2n$, so the condition $R_0 \neq 0$ implies that $\{h^0,h^1,\dots, h^{2n-1}\}$ is a minimal strong generating set for $\cA(n)^{GL(n)}$.

To prove this, we need to analyze the quantum corrections of $D_{I,J}$. Write $$R_{I,J} = R_n(I,J) H^m,\qquad m = 2n+\sum_{a=0}^n (i_a + j_a),$$ so that $R_n(I,J)$ denotes the coefficient of $H^m$ in $\text{pr}_m(D^2_{I,J})$. For $n=1$ and $I = (i_0, i_1)$, $J = (j_0, j_1)$ we have \begin{equation} \label{casenisone} R_1(I,J) = (-1)^{1+j_0 + j_1} \bigg(\frac{1}{2 + i_0 + j_0}  +  \frac{1}{ 2 + i_1 + j_0}  + \frac{1}{2 + i_0 + j_1} + \frac{1}{2 + i_1 + j_1} \bigg).\end{equation} Using the same method as the previous section, one can show that $R_n(I,J)$ satisfies the following recursive formula.

\begin{equation} \label{recursiongln}
R_n(I,J) =  \sum_{r=0}^n \bigg((-1)^{j_0} \bigg(\sum_k \frac{R_{n-1}(I_{r,k}, J')}{2+ i_k + j_0}\bigg) +  (-1)^{i_r}\sum_l \bigg(\frac{R_{n-1}(I_r, J'_l)}{2+ j_l + i_r} \bigg) \bigg).\end{equation} 
In this notation, $I_r = (i_0,\dots, \hat{i_r}, \dots, i_n)$ is obtained from $I$ by omitting $i_r$. For $k=0,\dots, n$ and $k\neq r$, $I_{r,k}$ is obtained from $I_r$ by replacing the entry $i_k$ with $i_k+i_r+j_0+2$. Similarly, $J' = (j_1,\dots, j_n)$ is obtained from $J$ by omitting $j_0$, and for $l=1,\dots, n$, $J'_l$ is obtained from $J'$ by replacing $j_l$ with $j_l + i_r + j_0+2$.

Suppose that all entries of $I$ and $J$ are even. Then for each $R_{n-1}(K,L)$ appearing in \eqref{recursiongln}, all entries of $K$ and $L$ are even. It is immediate from \eqref{casenisone} that $R_1(I,J)\neq 0$, and by induction on $n$, it follows from \eqref{recursiongln} that $R_n(I, J)$ is nonzero whenever $I$ and $J$ consist of even numbers. Specializing to the case $I = (0,\dots,0) = J$, we see that $R_0 \neq 0$, as desired. Finally, this implies
\begin{thm} $\cA(n)^{GL(n)}$ has a minimal strong generating set $\{h^0,h^1,\dots, h^{2n-1}\}$, and is therefore of type $\cW(2,3,\dots, 2n+1)$. \end{thm}

An immediate consequence is
\begin{thm} For all $m\geq 1$, $\cA(mn)^{GL(n)}$ has a minimal strong generating set $$\gamma^{j,k}_{0,l} = \sum_{i=1}^n :e^{i,j} \partial^l f^{i,k}:, \qquad 1\leq j\leq k\leq m,\qquad 0\leq l \leq 2n-1.$$ 
\end{thm} 
\begin{proof} The argument is similar to the proof of Theorem \ref{mn}, and is omitted.\end{proof}

\section{Character decompositions}\label{sec:ch}
A general result of Kac and Radul (see Section 1 of \cite{KR}) states that for an associative algebra $A$ and a Lie algebra $\gg$, every $(\gg, A)$-module $V$ with the properties that it is an irreducible $A$-module and a direct sum of a countable number of finite-dimensional irreducible $\gg$-modules
decomposes as
\[
V \cong \bigoplus_E \left( E\otimes V^E\right), 
\]
where the sum is over all irreducible $\gg$-modules $E$ and $V^E$ is an irreducible $A^\gg$-module. 
Kac and Radul then remark that the same statement is true if we replace $\gg$ by a group $G$ (see also \cite{DLMI} for similar results). The symplectic fermion algebra $\cA(mn)$ is simple and graded by conformal weight. Each graded subspace is finite-dimensional and $G$-invariant for any reductive $G\subset Sp(2mn)$. Hence the assumptions of \cite{KR} apply, so that 
\[
\cA(mn) \cong \bigoplus_E \left( E \otimes \cA(mn)^E\right), 
\]
where the sum is over all finite-dimensional irreducible representations of $G$, and $\cA(mn)^E$ is an irreducible representation of $\cA(mn)^G$. The purpose of this section is to find the characters of the $\cA(mn)^E$ for $G=Sp(2n)\times SO(n)$ and also for $G=GL(m)\times GL(n)$. We will use the denominator identities of finite-dimensional classical Lie superalgebras to solve this problem. We need some notation and results on Lie superalgebras. For this, we
use the book \cite{CW}, and the article \cite{KWII}.

Define the lattice 
\[
L_{m,n} := \delta_1 \ZZ\oplus\dots\oplus \delta_m \ZZ \oplus \epsilon_1\ZZ\oplus\dots\oplus\epsilon_n\ZZ
\]
of signature $(m, n)$, where the bilinear product is defined by
\[
\left( \delta_i, \delta_j\right) = \delta_{i, j},\qquad
\left( \epsilon_a, \epsilon_b\right) = -\delta_{a, b},\qquad
\left( \delta_i, \epsilon_b\right) = 0,
\]
for all $1\leq i, j \leq m$ and $1\leq a, b \leq n$. 
We restrict to the case $m>n$.
The root systems of various Lie superalgebras can be constructed as subsets of $L_{m, n}$.
We are interested in three examples. 
\begin{example} \label{ex:gl}
Let $\gg=\gg\gl\left(m|n\right)$. 
Then the root system $\Delta$ is the disjoint union of even and odd roots, $\Delta= \Delta_0\cup\Delta_1$, where 
\begin{equation}\nonumber
\begin{split}
\Delta_0 &= \left\{ \delta_i-\delta_j,  \ \epsilon_a-\epsilon_b\ | \ 1\leq i, j \leq m, \ 1\leq a, b \leq n\ \right\}\\
\Delta_1 &= \left\{ \pm\left(\delta_i-\epsilon_a\right)\ | \ 1\leq i \leq m, \ 1\leq a \leq n\ \right\}.
\end{split}
\end{equation}
A Lie superalgebra usually allows for inequivalent choices of positive roots and positive simple roots. We are interested in a choice of simple positive roots that has as many
odd isotropic roots as possible. This is if $m\leq n$
\begin{equation}\nonumber
\Pi = \left\{ \left(\epsilon_a-\delta_a\right), \ \delta_b-\epsilon_{b+1}\ |  1\leq a, b \leq m\ \right\}
\cup \left\{\epsilon_{i}-\epsilon_{i+1} \ | m+1\leq i\leq n\ \right\}
\end{equation}
so that positive roots are
\begin{equation}\nonumber
\begin{split}
\Delta^+ &= \Delta^+_0 \cup \Delta^+_1 \\
\Delta_0^+ &= \left\{ \delta_i-\delta_j,  \ \epsilon_a-\epsilon_b\ | \ 1\leq i < j \leq m, \ 1\leq a <b \leq n\ \right\}\\
\Delta^+_1 &=  \left\{ \epsilon_a-\delta_i,  \delta_j-\epsilon_b\ |\  1\leq a \leq n, \,  a\leq i \leq m, \, 1\leq j \leq m, \, j < b \leq n \ \right\}
\end{split}
\end{equation}
and 
\begin{equation}\label{eq:isotropic}
S = \left\{  \left(\epsilon_a-\delta_a\right)\ | \  1\leq a \leq m\ \right\} \subset \Pi
\end{equation}
is a maximal isotropic subset of simple positive odd roots. The case $m>n$ is analogous. 
For every positive root, define the Weyl reflection
\[
r_\alpha(x) := x-2\frac{\left(x, \alpha\right)}{\left(\alpha, \alpha\right)}\alpha.
\]
The Weyl group $W$ is then the group generated by the $r_\alpha$ for $\alpha$ in $\Delta_0^+$. 
Finally, let $W^\sharp$ be the subgroup of the Weyl group generated by the roots of $\gg\gl(m)$ if $m>n$, and otherwise the subgroup generated by 
the roots of $\gg\gl(n)$. 
\end{example}
\begin{example} 
Let $\gg=\gs\gp\go\left(2m|2n+1\right)$. 
Then the odd and even roots of the root system $\Delta = \Delta_0 \cup \Delta_1$ are
\begin{equation}\nonumber
\begin{split}
\Delta_0 &= \left\{ \pm\delta_i\pm\delta_j,  \ \pm\epsilon_a\pm \epsilon_b,\ \pm2\delta_p, \ \pm\epsilon_q  \ \right\},\qquad
\Delta_1 = \left\{ \pm\delta_p\pm\epsilon_q,\  \pm \delta_p\ \right\},
\end{split}
\end{equation}
where  $1\leq i < j \leq m, \ 1\leq a < b \leq n,\ 1\leq p\leq m, \ 1\leq q\leq n$.
A choice of simple positive roots that has as many
odd isotropic roots as possible is 
\begin{equation}\nonumber
\Pi = \left\{ (-1)^a \left(\epsilon_a-\delta_a\right),  \epsilon_b-\epsilon_{b+1},  \delta_i-\delta_{i+1}, \delta_m\, |  1\leq a, b \leq m,  b\, \text{odd}, 1\leq i\leq n,  i \, \ \text{even} \right\}
\end{equation}
so that 
positive roots $\Delta^+ = \Delta^+_0 \cup \Delta^+_1$ split into even and odd roots as 
\begin{equation}\nonumber
\begin{split}
\Delta_0^+ &= \left\{ \delta_i\pm\delta_j,  \ \epsilon_a\pm\epsilon_b,  \ 2\delta_p,\ \epsilon_q\  \right\}\\
\Delta^+_1 &=  \left\{ \epsilon_a\pm\delta_i,  \delta_j\pm\epsilon_b,\ \delta_p\  |\  1\leq a \leq n, \,  a\leq i \leq m, \, a\ \text{even}, \, 1\leq j \leq m, \, j\leq b \leq n, \, j\ \text{odd}\, \right\}
\end{split}
\end{equation}
where  $1\leq i < j \leq m, \ 1\leq a < b \leq n,\ 1\leq p\leq m, \ 1\leq q\leq n$ unless otherwise indicated. Then \eqref{eq:isotropic}
is a maximal isotropic subset of simple positive odd roots. The Weyl group $W$ is again the Weyl group of the even subalgebra. If $m>n$, $W^\sharp$ is defined to be the Weyl group of $\gs\gp(2m)$, and otherwise it is the Weyl group of $\gs\go(2n+1)$. 
\end{example}
\begin{example}\label{ex:spo} 
Let $\gg=\gs\gp\go\left(2m|2n\right)$. 
Then the root system $\Delta$ is the disjoint union of
\begin{equation}\nonumber 
\begin{split}
\Delta_0 &= \left\{ \pm\delta_i\pm\delta_j,  \ \pm\epsilon_a\pm \epsilon_b,\ \pm2\delta_p  \ \right\},\qquad
\Delta_1 = \left\{ \pm\delta_p\pm\epsilon_q\ \right\},
\end{split}
\end{equation}
where  $1\leq i < j \leq m, \ 1\leq a < b \leq n,\ 1\leq p\leq m, \ 1\leq q\leq n$.
In this case, a choice of simple positive roots that has as many
odd isotropic roots as possible is 
\begin{equation}\nonumber
\Pi = \left\{ (-1)^a \left(\epsilon_a-\delta_a\right),  \epsilon_b-\epsilon_{b+1},  \delta_i-\delta_{i+1}\, |  1\leq a, b \leq m,  b\, \text{odd}, 1\leq i\leq n,  i \, \text{even} \right\}
\end{equation}
so that 
positive roots are the disjoint union of 
\begin{equation}\nonumber
\begin{split}
\Delta_0^+ &= \left\{ \delta_i\pm\delta_j,  \ \epsilon_a\pm\epsilon_b,  \ 2\delta_p\  \right\}\\
\Delta^+_1 &=  \left\{ \epsilon_a\pm\delta_i,  \delta_j\pm\epsilon_b\ |\  1\leq a \leq n, \,  a\leq i \leq m, \, a\ \text{even}, \, 1\leq j \leq m, \, j\leq b \leq n, \, j\ \text{odd}\, \right\}
\end{split}
\end{equation}
where  $1\leq i < j \leq m, \ 1\leq a < b \leq n,\ 1\leq p\leq m, \ 1\leq q\leq n$ unless otherwise indicated. Then again \eqref{eq:isotropic}
is a maximal isotropic subset of simple positive odd roots. The Weyl group $W$ is as before the Weyl group of the even subalgebra. For $m>n$, $W^\sharp$ is the Weyl group of $\gs\gp(2m)$, and otherwise it is the Weyl group of $\gs\go(2n)$. 
\end{example}
The Weyl vector of a Lie superalgebra is defined to be $\rho:=\rho_0-\rho_1$
with even and odd Weyl vectors 
\[
\rho_0 = \frac{1}{2} \sum_{\alpha \in \Delta_0^+} \alpha,\qquad \rho_1 = \frac{1}{2} \sum_{\alpha \in \Delta_1^+} \alpha.
\]
We denote the root lattice of $\gg$ by $L$. Note that this lattice is spanned by the odd roots in our examples. 
In order to state the main result of this section, we need to define for every $\lambda$ in $L$ the set
\[
I_\lambda := \Big\{\ \left(\left(n_\alpha\right)_{\alpha\in\Delta^+_1\setminus S}, \left( m_\beta\right)_{\beta \in S}\right) \in \ZZ^{|\Delta^+_1|} \ \Big| \
\sum\limits_{\alpha\in \Delta^+_1\setminus S} \ \sum\limits_{\beta\in S} n_\alpha\alpha+m_\beta\beta =\lambda\ \Big\}
\]
as well as the partial theta function
\[
P_n(q):= q^{\frac{1}{2}\left(n+\frac{1}{2}\right)^2}\sum\limits_{m=0}^\infty (-1)^{m} q^{\frac{1}{2}\left(m^2-2m\left(n+\frac{1}{2}\right)\right)}.
\]
\begin{thm}\label{thm:ch}
Let $\gg_0=\gg\gl(m)\oplus \gg\gl(r)$ or $\gg_0=\gs\gp(2m)\oplus\gs\go(r)$ be the even subalgebra of
$\gg=\gg\gl(m|r)$, respectively $\gg=\gs\gp\go(2m|r)$. 
Let $P^+$ be the set of dominant weights of the even subalgebra $\gg_0$. 
Then the graded character of $\cA(mr)$ decomposes as 
\[
\ch[\cA(mr)] = \sum_{\Lambda\in L\cap P^+} \text{ch}_{\Lambda}  B_\Lambda.
\]
Here $\text{ch}_{\Lambda}$ is the irreducible highest-weight representation of $\gg_0$ of highest-weight $\Lambda$.
The branching function is
\[
B_\Lambda = \frac{1}{\eta(q)^{|\Delta^+_1|}}\frac{|W^\sharp|}{|W|}\sum_{\omega\in W}\epsilon(w) \sum_{\left((n_\alpha),(m_\beta)\right) \in I_{\omega(\Lambda+\rho_0)-\rho_0}}  q^{\frac{1}{2}\left(n_\alpha+\frac{1}{2}\right)^2} P_{m_\beta}(q)
\]
with the Dedekind eta function 
$\eta(q)=q^{\frac{1}{24}}\prod\limits_{n=1}^\infty (1-q^n)$.
\end{thm}

\begin{proof}
Observe that the generators $\left\{ e^i, f^i | i=1, \dots, mr \right\}$ of the symplectic fermion algebra $\cA(mr)$ carry the same representation of $\gg_0$ as the odd subalgebra of $\gg$. 
Hence we can write the graded character of $\cA(mr)$ as
\begin{equation}\nonumber
\ch[\cA(mr)]= q^{\frac{mr}{12}} \prod_{\alpha\in\Delta_1} \prod_{n=1}^\infty \left(1+e^{\alpha}q^n\right),
\end{equation}
where $\Delta_1$ is the odd root system of $\gg\in\{ \gs\gp\go(2m|r), \gg\gl(m|r)\}$. 
Using the odd Weyl vector, the character can be rewritten as 
\begin{equation}\nonumber
\begin{split}
\ch[\cA(mr)]&= e^{-\rho_1}q^{\frac{mr}{12}} \prod_{\alpha\in\Delta^+_1} \frac{e^{\frac{\alpha}{2}}}{\left(1+e^{-\alpha}\right)}\prod_{n=1}^\infty \left(1+e^{\alpha}q^n\right)\left(1+e^{-\alpha}q^{n-1}\right) 
= \frac{1}{e^{\rho_1}}\prod_{\alpha\in\Delta^+_1} \frac{\vartheta\left(e^\alpha ; q\right)}{\left(1+e^{-\alpha}\right)\eta(q)},
\end{split}
\end{equation}
where $\eta(q)$ is the Dedekind eta function and 
\begin{equation}\nonumber 
\vartheta(z; q) = \sum_{n\in\ZZ} z^{n+\frac{1}{2}}q^{\frac{1}{2}\left(n+\frac{1}{2}\right)^2}
\end{equation}
is a standard Jacobi theta function. 
The denominator identity of $\gg$ is \cite{G, KWII},
\begin{equation}\label{eq:denom}
\frac{e^{\rho_0}\prod\limits_{\alpha\in\Delta_0^+}(1-e^{-\alpha})}{e^{\rho_1}\prod\limits_{\alpha\in\Delta_1^+}(1+e^{-\alpha})} = \sum_{\omega\in W^\sharp} \epsilon(\omega) \omega\left(\frac{e^{\rho}}{\prod\limits_{\beta\in S}(1+e^{-\beta})} \right)
\end{equation}
where $W^\sharp$ is the Weyl group of the {\it larger} even subalgebra of $\gg$ as defined in Examples \ref{ex:gl}--\ref{ex:spo}. 
Here $\epsilon(\omega)$ denotes the signum of the Weyl reflection $\omega$. The set $S$ is the set of simple isotropic roots introduced in the examples. 
The left-hand side of \eqref{eq:denom} is invariant under $\epsilon(\omega)\omega$ for any reflection $\omega\in W$. Hence we can rewrite the denominator identity as
\[
\frac{e^{\rho_0}\prod\limits_{\alpha\in\Delta_0^+}(1-e^{-\alpha})}{e^{\rho_1}\prod\limits_{\alpha\in\Delta_1^+}(1+e^{-\alpha})} = \frac{|W^\sharp|}{|W|}\sum_{\omega\in W} \epsilon(\omega) \omega\left(\frac{e^{\rho}}{\prod\limits_{\beta\in S}(1+e^{-\beta})} \right).
\]
Observe that  $\prod\limits_{\alpha\in\Delta^+_1} \vartheta\left(e^\alpha ; q\right)$ is invariant under $W$, so we can rewrite the character as
\begin{equation}\nonumber
\ch[\cA(mr)] =  \frac{1}{\eta(q)^{|\Delta^+_1|}}\frac{|W^\sharp|}{|W|}\sum_{\alpha\in\Delta_1^+} \sum_{n_\alpha\in\ZZ}\sum_{\omega\in W} \epsilon(\omega) \omega\left(\frac{e^{\rho_0}e^{n_\alpha\alpha} q^{\frac{1}{2}\left(n_\alpha+\frac{1}{2}\right)^2}}{\prod\limits_{\beta\in S}(1+e^{-\beta})} \right)\left(e^{\rho_0}\prod\limits_{\alpha\in\Delta_0^+}(1-e^{-\alpha})\right)^{-1}.
\end{equation}
For $\lambda$ in $L$,  we define the set
\[
J_\lambda := \left\{\ \left(\left(n_\alpha\right)_{\alpha\in\Delta^+_1}, \left( m_\beta\right)_{\beta \in S}\right) \in \ZZ^{|\Delta^+_1|}\times \ZZ_{\geq 0}^{|S|} \ \Big| \
\sum_{\alpha\in \Delta^+_1} \sum_{\beta\in S} n_\alpha\alpha+m_\beta\beta =\lambda\ \right\}  
\]
and we also define
\[
\widetilde{\text{ch}}_{\lambda} := \frac{\sum\limits_{\omega\in W}\epsilon(\omega)  \omega\left(e^{\lambda+\rho_0}\right) }{e^{\rho_0}\prod\limits_{\alpha\in\Delta_0^+}(1-e^{-\alpha})}.
\]
Up to a sign this is the character $\text{ch}_\Lambda$ of the highest-weight representation of highest-weight $\Lambda$, where $\Lambda$ is the dominant weight
in the Weyl orbit of $\lambda$. Using this notation, the graded character becomes
\begin{equation}\nonumber
\begin{split}
\ch[\cA(mr)] &=  \frac{1}{\eta(q)^{|\Delta^+_1|}}\frac{|W^\sharp|}{|W|} \sum_{\lambda\in L}\widetilde{\text{ch}}_{\lambda} \sum_{\left((n_\alpha), (m_\beta)\right) \in J_\lambda} (-1)^{m_\beta} q^{\frac{1}{2}\left(n_\alpha+\frac{1}{2}\right)^2} \\
&=\frac{1}{\eta(q)^{|\Delta^+_1|}} \sum_{\Lambda\in L\cap P^+} \text{ch}_{\Lambda}  \sum_{\omega\in W} \frac{|W^\sharp|}{|W|}  \epsilon(w) \sum_{\left( (n_\alpha), (m_\beta)\right) \in J_{\omega(\Lambda+\rho_0)-\rho_0}} (-1)^{m_\beta} q^{\frac{1}{2}\left(n_\alpha+\frac{1}{2}\right)^2}
\end{split}
\end{equation}
where $P^+$ is the set of dominant weights of the even subalgebra of $\gg$. 
If $\alpha=\beta \in S$, then we can combine contributions as follows. Let $\ell_\beta=n_\beta+m_\beta$, then 
\[
\sum_{m_\beta=0}^\infty (-1)^{m_\beta} q^{\frac{1}{2}\left(n_\beta+\frac{1}{2}\right)^2}=
\sum_{m_\beta=0}^\infty (-1)^{m_\beta} q^{\frac{1}{2}\left(\ell_\beta-m_\beta+\frac{1}{2}\right)^2}=
q^{\frac{1}{2}\left(\ell_\beta+\frac{1}{2}\right)^2}\sum_{m_\beta=0}^\infty (-1)^{m_\beta} q^{\frac{1}{2}\left(m_\beta^2-2m_\beta\left(\ell_\beta+\frac{1}{2}\right)\right)}
\]
Inserting this terminates the proof. 
\end{proof}

\begin{cor}
For $G=Sp(2n)$, the branching function is
\[
B_\Lambda(q)=  \frac{1}{\eta(q)^n}\sum\limits_{\omega \in W}\epsilon(\omega)q^{\frac{1}{2}\left(\omega(\Lambda+\rho_0)-\rho, \omega(\Lambda+\rho_0)-\rho\right)}.
\]
\end{cor}
\begin{proof}
We have $|\Delta^+_1|=n$ and $W^\sharp=W$. The set $S$ is empty and the set $I_{\Lambda}$ ($\Lambda$ in $L$) is just the one element set $\{\Lambda\}$. For $\Lambda= m_1\delta_1+\dots +m_n \delta_n$, we have
\[
\sum_{i=1}^n \left(m_i+\frac{1}{2}\right)^2 = \left(\Lambda+\rho_1, \Lambda+\rho_1\right),
\]
so that the branching function simplifies as claimed. 
\end{proof}
\begin{cor}
$\cA(n)^{Sp(2n)}$ has graded character $$\ch[\cA(n)^{Sp(2n)}]  = q^{\frac{n}{12}} \sum_{m\geq 0}\text{dim}(\cA(n)^{Sp(2n)}[m]) q^m = q^{\frac{n}{12}} \prod_{i=1}^n \prod_{k\geq 0} \frac{1}{1-q^{2i +k}}.$$
\end{cor}
\begin{proof}
Using the denominator identity of $\gs\gp(2n)$, we get
\begin{equation*}
\begin{split}
B_0 &=  \frac{1}{\eta(q)^n}\sum\limits_{\omega \in W}\epsilon(\omega)q^{\frac{1}{2}\left(\omega(\rho_0)-\rho, \omega(\rho_0)-\rho\right)} 
= \frac{q^{\frac{1}{2}\left((\rho_0, \rho_0)+(\rho, \rho) \right)}}{\eta(q)^n}\sum\limits_{\omega \in W}\epsilon(\omega)q^{-\left(\omega(\rho_0),\rho\right)} \\
&= \frac{q^{\frac{1}{2}\left(\rho_1, \rho_1\right)}}{\eta(q)^n} \prod_{\alpha\in \Delta_0^+}\left(1-q^{\left(\alpha,\rho\right)}\right).
\end{split}
\end{equation*}
The claim follows from the products of root vectors
\[
\left(\rho, \delta_i+\delta_j\right) = 2n+1-i-j, \qquad 
\left(\rho, \delta_i-\delta_j\right) = j-i, \qquad
\left(\rho, 2\delta_i\right) = 2n+1-2i,  
\]
and the norm $\left(\rho_1, \rho_1\right) =\frac{n}{4}$ of the odd Weyl vector. 
\end{proof}

\begin{cor} \label{free} $\cA(n)^{Sp(2n)}$ is {\it freely} generated by $\{j^0, j^2, \dots, j^{2n-2}\}$. In other words, there are no normally ordered polynomial relations among these generators and their derivatives.\end{cor}

We turn to the case of $G=GL(m)$. Note that $\cA(mn)^{GL(n)\times GL(m)}=\cA(mn)^{GL(n)\times SL(m)}$ since one of the $GL(1)$ factors acts trivially.
For the case $\gg\gl(m|1)$, we take as positive odd roots the set $\{ \epsilon-\delta_1, \delta_2-\epsilon, \delta_3-\epsilon, \dots, \delta_m-\epsilon\}$
and we take $S=\{ \epsilon-\delta_1\}$. 
The set $I_\lambda$ has only one element, and it is a short computation to obtain the following explicit expression
for the branching functions.
\begin{cor}
Let $G=GL(m)$ and $\Lambda=\Lambda_1\delta_1 +\dots + \Lambda_m \delta_m$. Then the branching function is
\[
B_\Lambda(q)=  \frac{q^{-\frac{(m-2)^2}{4}}}{\eta(q)^m}\sum\limits_{\omega \in W}\epsilon(\omega)
q^{\frac{1}{2}\left(\omega(\Lambda+\rho_0)-\rho, \omega(\Lambda+\rho_0)-\rho\right)}\sum_{r=0}^\infty (-1)^r 
q^{\frac{r(r-1)}{2}} q^{\left(\delta_1, \omega(\Lambda+\rho_0)-\rho_0\right) (r+1)}.
\] 
\end{cor}

\begin{remark}
Asymptotic properties of partial theta functions, like the one appearing in the character decomposition, have been studied recently in \cite{BM, CMW}. Especially in both works with different methods it has been found that
\[
\lim_{q\rightarrow 1} P_{n}(q) =\frac{1}{2},\qquad\qquad \text{for all} \ n\in\ZZ.
\] 
We also remark that these partial theta functions are not mock modular forms, but they are quantum modular forms as defined by Zagier \cite{Za}.
\end{remark}

\begin{cor}
$\cA(n)^{GL(n)}$ is not freely generated by $\{h^0, h^1, \dots, h^{2n-1}\}$.
\end{cor}
\begin{proof}
By the previous remark and corollary we get 
\[
\lim_{q\rightarrow 1} \Big|\eta(q)^n B_0(q) \Big| \leq \sum_{\omega\in W} \frac{1}{2} = \frac{n!}{2}.
\]
On the other hand, if $\cA(n)^{GL(n)}$ is freely generated its character would be 
\[
\prod_{i=2}^{2n+1} \prod_{k\geq 0} (1-q^{i+k})^{-1}.
\]
But the latter has a different behavior as $q$ approaches one, namely
\[
\lim_{q\rightarrow 1} \left(\eta(q)^n\prod_{i=2}^{2n+1} \prod_{k\geq 0} (1-q^{i+k})^{-1}\right) =\infty,
\]
so that the two cannot coincide.
\end{proof}

\section{The Hilbert problem for $\cA(n)$}

For a simple, strongly finitely generated vertex algebra $\cV$ and a reductive group $G\subset \text{Aut}(\cV)$, the {\it Hilbert problem} asks whether $\cV^G$ is strongly finitely generated. In this section we solve this problem affirmatively for $\cV = \cA(n)$. The approach is similar to the one used for the $\beta\gamma$-system and free fermion algebra in \cite{LV}, and some details are omitted.

Recall from \cite{Zh} that the {\it Zhu functor} assigns to a vertex algebra $\cV = \bigoplus_{n\in\mathbb{Z}} \cV_n$ an associative algebra $A(\cV)$ and a surjective linear map $$\pi_{\text{Zhu}}:\cV\ra A(\cV),\qquad a \mapsto [a].$$ A $\mathbb{Z}_{\geq 0}$-graded module $M = \bigoplus_{n\geq 0} M_n$ over $\cW$ is called {\it admissible} if for every $a\in\cV_m$, $a(n) M_k \subset M_{m+k -n-1}$, for all $n\in\mathbb{Z}$. Given $a\in\cV_m$, $a(m-1)$ acts on each $M_k$. Then $M_0$ is an $A(\cV)$-module with action $[a]\mapsto a(m-1) \in \text{End}(M_0)$, and $M\mapsto M_0$ provides a one-to-one correspondence between irreducible, admissible $\cV$-modules and irreducible $A(\cV)$-modules. If $A(\cV)$ is commutative, all its irreducible modules are one-dimensional. The corresponding $\cV$-modules $M$ are therefore cyclic and will be called {\it highest-weight modules}.

\begin{lemma} \label{abelianzhu} For $n\geq 1$, $A(\cA(n)^{Sp(2n)})$ is commutative.\end{lemma}

\begin{proof} Since $\cA(n)^{Sp(2n)}$ is strongly generated by $j^0, j^2,\dots, j^{2n-2}$, $A(\cA(n)^{Sp(2n)})$ is generated by $\{a^0, a^2,\dots, a^{2n-2}\}$ where $a^{2m} =  \pi_{\text{Zhu}}(j^{2m})$. The commutativity of $A(\cA(n)^{Sp(2n)})$ for $n=1$ is clear since there is only one generator, and it is also clear for $n=2$ since $a^0$ is central, and hence commutes with $a^2$. It follows that for $n=2$, $[a^{2i}, a^{2j}] = 0$ where $i,j\geq 0$. Since the nonconstant terms appearing in the operator product $a^{2i}(z) a^{2j}(w)$ are independent of $n$, it follows that $[a^{2i}, a^{2j}] = 0$ for all $n$. \end{proof}

\begin{cor} \label{classificationmod} For $n\geq 1$, $A(\cA(n)^{Sp(2n)})\cong \mathbb{C}[a^0, a^2,\dots, a^{2n-2}]$, so the irreducible, admissible $\cA(n)^{Sp(2n)}$-modules are all highest-weight modules, and are parametrized by the points in $\mathbb{C}^n$.\end{cor}

\begin{proof} This is immediate from Lemma \ref{abelianzhu} and Corollary \ref{free}. \end{proof}

\begin{remark} $A(\cA(n)^{GL(n)})$ is also commutative, but we shall not need this fact. \end{remark}

Recall that $\cA(n)\cong \text{gr}(\cA(n))$ as $Sp(2n)$-modules, and $$\text{gr}(\cA(n)^G )\cong (\text{gr}(\cA(n))^G \cong (\bigwedge \bigoplus_{k\geq 0} U_k)^G = R$$ as supercommutative algebras, where $U_k\cong \mathbb{C}^{2n}$. For all $p\geq 1$, $GL(p)$ acts on $\bigoplus_{k =0}^{p-1} U_k $ and commutes with the action of $G$. The inclusions $GL(p)\hookrightarrow GL(p+1)$ induce an action of $GL(\infty) = \lim_{p\ra \infty} GL(p)$ on $\bigoplus_{k\geq 0} U_k$, so $GL(\infty)$ acts on $\bigwedge \bigoplus_{k\geq 0} U_k$ and commutes with the action of $G$. Therefore $GL(\infty)$ acts on $R$ as well. Elements $\sigma \in GL(\infty)$ are known as {\it polarization operators}, and $\sigma f$ is known as a polarization of $f\in R$. 

\begin{thm} \label{weylfinite} There exists an integer $m\geq 0$ such that $R$ is generated by the polarizations of any set of generators for $(\bigwedge \bigoplus_{k = 0}^{m} U_k)^G$. Since $G$ is reductive, $(\bigwedge \bigoplus_{k = 0} ^{m} U_k)^G$ is finitely generated, so there exists a finite set $\{f_1,\dots, f_r\}$, whose polarizations generate $R$. \end{thm}

\begin{proof} It suffices to show that the degrees of the generators of $(\bigwedge \bigoplus_{k \geq 0} U_k)^G$ have an upper bound $d$, since we can then take $m=d$. By Theorem 2.5A of \cite{We}, $(\text{Sym} \bigoplus_{k\geq 0} U_k )^G$ is generated by the polarizations of any set of generators of $(\text{Sym} \bigoplus_{k=0}^{2n-1} U_k )^G$. Since $G$ is reductive $(\text{Sym} \bigoplus_{k=0}^{2n-1} U_k )^G$ is finitely generated, and since polarization preserves degree, the degrees of the generators of $(\text{Sym} \bigoplus_{k\geq 0} U_k )^G$ have an upper bound $d$.

For any vector space $W$, recall that $\bigwedge^r W$ is the quotient of the $r^{\text{th}}$ tensor power $T^r W$ under the antisymmetrization map. We may regard $T^r W$ as the homogeneous subspace $$\text{Sym}^1(W_1)\otimes \cdots \otimes \text{Sym}^1(W_r) \subset \text{Sym} \bigoplus_{j=1}^r W_j,\qquad W_j \cong W$$ of degree $1$ in each $W_j$. Suppose that $G$ acts on $W$, and that the generators of $\text{Sym} \bigoplus_{j\geq 1} W_j$ have upper bound $d$. The symmetric group $\mathfrak{S}_r$ acts on $T^r W$ by permuting the factors, and $(T^r W)^G$ is generated by the elements of $(T^j W)^G$ for $j\leq d$, together with their translates under $\mathfrak{S}_r$. By taking $W = \bigoplus_{k \geq 0} U_k$ and applying the antisymmetrization map, it follows that $(\bigwedge \bigoplus_{k \geq 0} U_k)^G$ is generated by the elements of degree at most $d$.
\end{proof}

\begin{cor} \label{ordfg} $\cA(n)^G$ has a strong generating set which lies in the direct sum of finitely many irreducible $\cA(n)^{Sp(2n)}$-submodules of $\cA(n)$.
\end{cor}

\begin{proof} This is a straightforward consequence of Corollary \ref{classificationmod}, Theorem \ref{weylfinite}, and the fact (see Remark 1.1 of \cite{KR}) that $\cA(n)$ has a Howe pair decomposition \begin{equation}\label{dlmdecomp} \cA(n) \cong \bigoplus_{\nu\in H} L(\nu)\otimes M^{\nu},\end{equation} where $H$ indexes the irreducible, finite-dimensional representations $L(\nu)$ of $Sp(2n)$, and the $M^{\nu}$ are inequivalent, irreducible highest-weight $\cA(n)^{Sp(2n)}$-modules. The argument is similar to the proof of Lemma 2 of \cite{LII} and is omitted. \end{proof}

Using the strong finite generation of $\cA(n)^{Sp(2n)}$ together with some finiteness properties of the irreducible, highest-weight $\cA(n)^{Sp(2n)}$-modules appearing in $\cA(n)$, we obtain the following result. This is analogous to Lemma 9 of \cite{LII}, and the proof is omitted.

\begin{lemma} \label{sixth} Let $\cM$ be an irreducible, highest-weight $\cA(n)^{Sp(2n)}$-submodule of $\cA(n)$. Given a subset $S\subset \cM$, let $\cM_S\subset \cM$ denote the subspace spanned by elements of the form $$:\omega_1(z)\cdots \omega_t(z) \alpha(z):,\qquad \omega_j(z)\in \cA(n)^{Sp(2n)},\qquad \alpha(z)\in S.$$ There exists a finite set $S\subset \cM$ such that $\cM = \cM_S$.\end{lemma} 

\begin{thm} \label{sfg} For any reductive group $G$ of automorphisms of $\cA(n)$, $\cA(n)^G$ is strongly finitely generated.\end{thm}

\begin{proof} By Corollary \ref{ordfg}, we can find $f_1(z),\dots, f_r(z) \in \cA(n)^G$ such that $\text{gr}(\cA(n))^G$ is generated by the corresponding polynomials $f_1,\dots, f_r\in \text{gr}(\cA(n))^G$, together with their polarizations. We may assume that each $f_i(z)$ lies in an irreducible, highest-weight $\cA(n)^{Sp(2n)}$-module $\cM_i$ of the form $L(\nu)_{\mu_0}\otimes M^{\nu}$, where $L(\nu)_{\mu_0}$ is a trivial, one-dimensional $G$-module. Furthermore, we may assume that $f_1(z),\dots, f_r(z)$ are highest-weight vectors for the action of $\cA(n)^{Sp(2n)}$. For each $\cM_i$, choose a finite set $S_i \subset \cM_i$ such that $\cM_i = (\cM_i)_{S_i}$, using Lemma \ref{sixth}. Define $$ S=\{j^0,j^2,\dots, j^{2n-2} \} \cup \big(\bigcup_{i=1}^r S_i \big).$$ Since $\{j^0,j^2,\dots, j^{2n-2}\}$ strongly generates $\cA(n)^{Sp(2n)}$, and $\bigoplus_{i=1}^r \cM_i$ contains 
a strong generating set for $\cA(n)^G$, it follows that $S$ is a strong, finite generating set for $\cA(n)^G$. \end{proof}

\end{document}